\documentclass[11pt, a4paper,leqno]{article}

\usepackage[OT1]{fontenc}
\usepackage{amsthm,amsmath, amssymb, amsopn, amsfonts}
\usepackage[colorlinks,citecolor=blue,urlcolor=blue]{hyperref}
\usepackage{footnote}
\usepackage{graphicx}
\usepackage{tikz}
\usetikzlibrary{decorations.fractals}
\usetikzlibrary{patterns,arrows,shapes,decorations.pathreplacing}
\usepackage{pgfplots}
\usepackage{stmaryrd} 
\usepackage[margin=1in]{geometry}

\numberwithin{equation}{section}
\theoremstyle{plain}
\newtheorem{Definition}{Definition}[section]
\newtheorem{Remark}{Remark}[section]
\newtheorem{Theorem}{Theorem}[section]
\newtheorem{Lemma}{Lemma}[section]
\newtheorem{Proposition}{Proposition}[section]
\newtheorem{Corollary}{Corollary}[section]
\newtheorem{Assumption}{Assumption}[section]
\newtheorem{Example}{Example}[section]
\def \E{\mathbb{E}}
\def \F{\mathbb{F}}
\def \N{\mathbb{N}}
\def \P{\mathbb{P}}
\def \Q{\mathbb{Q}}
\def \R{\mathbb{R}}
\def \X{\mathbb{X}}
\def \Bc{{\mathcal B}}
\def \Cc{{\mathcal C}}
\def \Ec{{\mathcal E}}
\def \Fc{{\mathcal F}}
\def \Lc{{\mathcal L}}

\def \Tc{{\mathcal T}}
\def \eps{\varepsilon}
\def \x{\times}
\def\esssup{\mathop{\rm ess\, sup}} 

\def \indi{\mathbf{1}}
\def \brace#1{\left\{#1\right\}}
\def \And{\;\mbox{ and }\;}

\DeclareMathOperator{\sech}{sech}

\title{Time-consistent stopping under decreasing impatience\thanks{For thoughtful advice and comments, we thank Erhan Bayraktar, Ren\'{e} Carmona, Samuel Cohen, Ivar Ekeland, Paolo Guasoni, Jan Ob\l\'{o}j, Traian Pirvu, Ronnie Sircar, and Xunyu Zhou, and seminar participations at Florida State University, Princeton University, and University of Oxford. Special gratitude goes to Erhan Bayraktar for bringing this problem to the first author's attention. Special gratitude also goes to Traian Pirvu for introducing the authors to know each other.
Y.-J. Huang is partially supported by the University of Colorado (11003573).
A. Nguyen-Huu is partially supported by the \textit{Energy and Prosperity} Chair.}}

\author{Yu-Jui Huang\thanks{Department of Applied Mathematics, University of Colorado, Boulder, CO 80309, USA. Corresponding Author : yujui.huang@colorado.edu} \and Adrien Nguyen-Huu\thanks{LAMETA, Universit\'e de Montpellier, Montpellier, France.}}
\date{}

\begin{document}
\maketitle

\begin{abstract}
Under non-exponential discounting, we develop a dynamic theory for stopping problems in continuous time. 
Our framework covers discount functions that induce decreasing impatience. 
Due to the inherent time inconsistency, 
we look for equilibrium stopping policies, formulated as fixed points of an operator. 
Under appropriate conditions, fixed-point iterations converge to equilibrium stopping policies. 
This iterative approach corresponds to the hierarchy of strategic reasoning in Game Theory, and provides ``agent-specific'' results: it assigns one specific equilibrium stopping policy to each agent according to her initial behavior. In particular, it leads to a precise mathematical connection between the naive behavior and the sophisticated one.  Our theory is illustrated in a real options model.
\end{abstract}

\textbf{Keywords:} time inconsistency \and optimal stopping \and hyperbolic discounting \and decreasing impatience \and subgame-perfect Nash equilibrium

\textbf{JEL:} C61; D81; D90; G02

\textbf{2010 Mathematics Subject Classification:} 60G40; 91B06

\section{Introduction} \label{sec:intro}

Time inconsistency is known to exist in stopping decisions, such as casino gambling in \cite{Barberis12} and \cite{ebert2015until}, optimal stock liquidation in \cite{XZ13}, and real options valuation in \cite{grenadier2007investment}. A general treatment, however, has not been proposed in continuous-time models. In this article, we develop a dynamic theory for time-inconsistent stopping problems in continuous time, under non-exponential discounting. 
In particular, we focus on log sub-additive discount functions (Assumption \ref{asm:DI}), which capture \textit{decreasing impatience}, an acknowledged feature of empirical discounting in Behavioral Economics; see e.g. \cite{Thaler81}, \cite{LT89}, and  \cite{LP92}. Hyperbolic and quasi-hyperbolic discount functions are special cases under our consideration.

The seminal work Strotz \cite{Stroz1956myopia} identifies three types of agents under time inconsistency -- the \emph{naive}, the \emph{pre-committed}, and the \emph{sophisticated}. Among them, only the sophisticated agent  takes the possible change of future preferences seriously, and works on {\it consistent planning}:  she aims to find a strategy that once being enforced over time, none of her future selves would want to deviate from it. How to precisely formulate such a {\it sophisticated strategy} had been a challenge in continuous time. For stochastic control, Ekeland and Lazrak \cite{EL06} resolved this issue by defining sophisticated controls as subgame-perfect Nash equilibria in a continuous-time inter-temporal game of multiple selves. This has aroused vibrant research on time inconsistency in mathematical finance; see e.g.  \cite{Ekeland2008investment}, \cite{Ekeland2012time}, \cite{Hu2012time}, \cite{Yong2012time}, \cite{bjork2014mean}, \cite{dong2014time}, \cite{Bjork2014theory-discrete}, and \cite{BKM16}. There is, nonetheless, no equivalent development for stopping problems.


This paper contributes to the literature of time inconsistency in three ways. First, we provide a precise definition of sophisticated stopping policy (or, equilibrium stopping policy) in continuous time (Definition~\ref{def:equilibrium}). Specifically, we introduce the operator $\Theta$ in \eqref{Theta}, which describes the game-theoretic reasoning of a sophisticated agent. Sophisticated policies are formulated as fixed points of $\Theta$, which connects to the concept of subgame-perfect Nash equilibrium invoked in \cite{EL06}.  

Second, we introduce a new, iterative approach for finding equilibrium strategies. For any initial stopping policy $\tau$, we apply the operator $\Theta$ to $\tau$ repetitively until it converges to an equilibrium stopping policy. Under appropriate conditions, this fixed-point iteration indeed converges (Theorem~\ref{thm:main 2}), which is the main result of this paper. Recall that the standard approach for finding equilibrium strategies in continuous time is solving a system of non-linear equations, as proposed in \cite{Ekeland2008investment} and \cite{BKM16}. Solving this system of equations is difficult; and even when it is solved (as in the special cases in \cite{Ekeland2008investment} and \cite{BKM16}), we only obtain one particular equilibrium, and it is unclear how other equilibrium strategies can be found. Our iterative approach can be useful here: we find different equilibria simply by starting the fixed-point iteration with different initial strategies $\tau$. In some cases, we are able to find {\it all} equilibria; see Proposition~\ref{prop:entire Ec}. 

Third, when an agent starts to do game-theoretic reasoning and look for equilibrium strategies, she is not satisfied with an arbitrary equilibrium. Instead, she works on improving her initial strategy to turn it into an equilibrium. This improving process is absent from \cite{EL06}, \cite{Ekeland2008investment}, \cite{BKM16}, and subsequent research, although well-known in Game Theory as the hierarchy of strategic reasoning in \cite{Stahl93} and \cite{SW94}. Our iterative approach specifically represents this improving process: for any initial strategy $\tau$, each application of $\Theta$ to $\tau$ corresponds to an additional level of strategic reasoning. As a result, the iterative approach complements the existing literature of time inconsistency in that it not only facilitates the search for equilibrium strategies, but provides ``agent-specific'' equilibria: it assigns one specific equilibrium to each agent according to her initial behavior.

Upon completion of our paper, we noticed the recent work Pedersen and Peskir \cite{pedersen2016optimal} on mean-variance optimal stopping. They introduced ``dynamic optimality'' to deal with time inconsistency. As explained in detail in \cite{pedersen2016optimal}, this new concept is different from {\it consistent planning} in Strotz \cite{Stroz1956myopia}, and does not rely on game-theoretic modeling. Therefore, our equilibrium stopping policies are different from their dynamically optimal stopping times. That being said, a few connections between our paper and \cite{pedersen2016optimal} do exist, as pointed out in Remarks~\ref{rem:static optimality}, \ref{rem:dynamic optimality}, and \ref{rem:PP and ours}.    

The paper is organized as follows.
In Section \ref{sec:motivation}, we introduce the setup of our model, and demonstrate time inconsistency in stopping decisions through examples. 
In Section \ref{sec:equilibrium},
we formulate the concept of equilibrium for stopping problems in continuous time, search for equilibrium strategies via fixed-point iterations, and establish the required convergence result. 
Section \ref{sec:examples} illustrates our theory thoroughly in a real options model. 
Most of the proofs are delegated to appendices.


\section{Preliminaries and Motivation}
\label{sec:motivation}


Consider the canonical space $\Omega:= \{\omega\in C([0,\infty);\R^d):\omega_0=0\}$. 
Let $\{W_t\}_{t\ge 0}$ be the coordinate mapping process $W_t(\omega) = \omega_t$, 
and $\F^W =\{\Fc^W_s\}_{s\ge 0}$ be the natural filtration generated by $W$. 
Let $\P$ be the Wiener measure on $(\Omega,\Fc^W_\infty)$, 
where $\Fc^W_\infty:= \bigcup_{s\ge 0} \Fc^W_s$. 
For each $t\ge 0$, we introduce the filtration $\F^{t,W} =\{\Fc^{t,W}_s\}_{s\ge 0}$ with 
\begin{equation*}
\Fc^{t,W}_s = \sigma(W_{u\vee t}-W_t:0\le u\le s),
\end{equation*}
and let $\F^t =\{\Fc^t_s\}_{s\ge 0}$ be the $\P$-augmentation of $\F^{t,W}$. We denote by $\Tc_t$ the collection of all $\F^t$-stopping times $\tau$ with $\tau\ge t$ a.s. For the case where $t=0$, we simply write $\F^0=\{\Fc^0_s\}_{s\ge 0}$ as $\F_s = \{\Fc_s\}_{s\ge 0}$, and $\Tc_0$ as $\Tc$.

\begin{Remark}
	For any $0\le s\le t$, $\Fc^t_s$ is the $\sigma$-algebra generated by only the $\P$-negligible sets. Moreover, for any $s,t\ge 0$, $\Fc^t_s$-measurable random variables are independent of $\Fc_t$; see Bouchard and Touzi \cite[Remark 2.1]{BT11} for a similar set-up.
\end{Remark}

Consider the space $\X:= [0,\infty)\times\R^d$, equipped with the Borel $\sigma$-algebra $\Bc(\X)$. Let $X$ be a continuous-time Markov process given by $X_s:= f(s,W_s)$, $s\ge 0$, for some measurable function $f:\X\mapsto\R$. Or, more generally, for any $\tau\in\Tc$ and $\R^d$-valued $\Fc_\tau$-measurable $\xi$, let $X$ be the solution to the stochastic differential equation
\begin{equation}\label{SDE}
dX_t = b(t,X_t) dt + \sigma (t,X_t) dW_t\quad\hbox{for}\ t\ge \tau,\quad\quad \hbox{with}\   X_\tau = \xi\ \hbox{a.s.}
\end{equation}
We assume that $b:\X\mapsto\R$ and $\sigma:\X\mapsto\R$ satisfy Lipschitz and linear growth conditions in $x\in\R^d$, uniformly in $t\in[0,\infty)$. Then, for any $\tau\in\Tc$ and $\R^d$-valued $\Fc_\tau$-measurable $\xi$ with $\E[|\xi|^2]<\infty$, \eqref{SDE} admits a unique strong solution.

For any $(t,x)\in\X$, we denote by $X^{t,x}$ the solution to \eqref{SDE} with $X_t=x$, and by $\E^{t,x}$ the expectation conditioned on $X_t=x$.

\subsection{Classical Optimal Stopping}\label{subsec:classical optimal stopping}

Consider a {\it payoff function} $g:\R^d\mapsto\R$, assumed to be nonnegative and  continuous, and a {\it discount function} $\delta:\R_+\mapsto [0,1]$, assumed to be continuous, decreasing, and satisfy $\delta(0)=1$. Moreover, we assume that
\begin{equation}\label{dominated}
\E^{t,x}\bigg[\sup_{t\le s\le \infty}\delta(s-t) g(X^{}_s)\bigg]<\infty, \quad \forall (t,x)\in \X,
\end{equation}
where we interpret $\delta(\infty-t)g(X^{t,x}_\infty) := \limsup_{s\to\infty}\delta(s-t)g(X^{t,x}_s)$; this is in line with  Karatzas and Shreve \cite[Appendix D]{KS-book-98}. 
Given $(t,x)\in \X$, classical optimal stopping concerns if 
there is a $\tau\in \Tc_t$ such that the expected discounted payoff
\begin{equation} \label{objective function}
J(t,x;\tau):=\E^{t,x}\left[\delta(\tau-t)g(X^{}_\tau) \right]
\end{equation} 
can be maximized. 
The associated value function
\begin{equation}\label{value function}
v(t,x) := \sup\limits_{\tau\in \Tc^{}_t} J(t,x; \tau)
\end{equation}
has been widely studied, and the existence of an optimal stopping time is affirmative. The following is a standard result taken from \cite[Appendix D]{KS-book-98} and \cite[Chapter I.2]{Ps-book-06}. 

\begin{Proposition}\label{prop:standard result}
	For any $(t,x)\in\X$, let $\{Z^{t,x}_s\}_{s\ge t}$ be a right-continuous process with 
	\begin{equation}\label{Snell}
	Z^{t,x}_s(\omega) =\esssup_{\tau\in\Tc_s} \E^{s,X^{t,x}_{s}(\omega)}[\delta(\tau^{}-t) g(X^{}_{\tau^{}})]\quad \hbox{a.s.}\quad \forall s\ge t,
	\end{equation}
	and define $\widetilde\tau_{t,x}\in\Tc^{}_t$ by 
	\begin{align}
	\widetilde{\tau}_{t,x} &:= \inf \left\{ s\ge t ~:~ \delta(s-t)g(X^{t,x}_s) = Z^{t,x}_s\right\}.\label{naive}
	\end{align}
	Then, $\widetilde{\tau}_{t,x}$ is an optimal stopping time of \eqref{value function}, i.e.
	\begin{equation}\label{J=esssup J}
	J(t,x; \widetilde{\tau}_{t,x}) = \sup\limits_{\tau\in \Tc^{}_{t}} J(t,x; \tau) .
	\end{equation}
	Moreover, $\widetilde{\tau}_{t,x}$ is the smallest, if not unique, optimal stopping time.
\end{Proposition}

\begin{Remark}\label{rem:static optimality}
	The classical optimal stopping problem \eqref{value function} is {\it static} in the sense that it involves only the preference of the agent at time $t$. Following the terminology of Definition 1 in Pedersen and Peskir \cite{pedersen2016optimal}, $\widetilde\tau_{t,x}$ in \eqref{naive} is ``statically optimal''.
\end{Remark}

\subsection{Time Inconsistency}
\label{subsec:time inconsistency}

Following Strotz \cite{Stroz1956myopia},
a naive agent solves the classical problem \eqref{value function} repeatedly at every moment as time passes by. That is, given initial $(t,x)\in\X$, the agent solves 
\begin{equation*}
\sup_{\tau\in\Tc_s} J(s,X^{t,x}_s;\tau)\quad \hbox{at every moment}\ s\ge t. 
\end{equation*}
By Proposition~\ref{prop:standard result}, 
the agent at time $s$ intends to employ the
stopping time $\widetilde\tau_{s,X^{t,x}_s}\in\Tc_s$, for all $s\ge t$. 
{This raises the question of} whether optimal stopping times obtained at different moments, 
$\widetilde\tau_{t,x}$ and $\widetilde\tau_{t',X^{t,x}_{t'}}$ with $t'>t$, are consistent with each other.

\begin{Definition} [Time Consistency]
	\label{def:time consistency}
	The problem \eqref{value function} is time-consistent if for any $(t,x)\in\X$ and $s>t$, 
	$
	\widetilde{\tau}_{t,x} (\omega)= \widetilde{\tau}_{s,X^{t,x}_s(\omega)}(\omega)\ \ \hbox{for a.e. $\omega\in\{\widetilde\tau_{t,x}\ge s\}$}.
	$
	We say the problem \eqref{value function} is time-inconsistent if the above does not hold.
\end{Definition}

In the classical literature of Mathematical Finance, the discount function usually takes the form $\delta(s) = e^{-\rho s}$ for some $\rho \ge 0$. This already guarantees time consistency of \eqref{value function}. To see this, first observe the identity
\begin{equation}\label{identity}
\delta(s)\delta(t) =\delta(s+t)\quad \forall s,t\ge 0.
\end{equation}
Fix $(t,x)\in\X$ and pick $t'>t$ such that $\P[\widetilde\tau_{t,x}\ge t']>0$. For a.e. $\omega\in\{\widetilde\tau_{t,x}\ge t'\}$, set $y:=X^{t,x}_{t'}(\omega)$. We observe from \eqref{naive}, \eqref{Snell}, and $X^{t,x}_s(\omega) = X^{t',y}_s(\omega)$ that
\begin{align*}
\widetilde\tau_{t,x}(\omega) &= \inf \left\{ s\ge t' : \delta(s-t)g(X^{t',y}_s(\omega)) \ge \esssup_{\tau\in \Tc_s}\E^{s,X^{t',y}_s(\omega)}[\delta(\tau-t)g(X^{}_\tau)]\right\},\\
\widetilde\tau_{t',y}(\omega)  &= \inf \left\{ s\ge t' : \delta(s-t')g(X^{t',y}_s(\omega)) \ge \esssup_{\tau\in \Tc_s}\E^{s,X^{t',y}_s(\omega)}[\delta(\tau-t')g(X^{}_\tau)]\right\}.
\end{align*}
Then \eqref{identity} guarantees $\widetilde\tau_{t,x}(\omega)= \widetilde\tau_{t',y}(\omega)$, as $\frac{\delta(\tau-t)}{\delta(s-t)}=\frac{\delta(\tau-t')}{\delta(s-t')}=\delta(\tau-s)$. 
For non-exponential discount functions, the identity \eqref{identity} no longer holds, and the problem \eqref{value function} is in general time-inconsistent.

\begin{Example}[Smoking Cessation] \label{eg:smoking} 
	Suppose a smoker has a fixed lifetime $T>0$. Consider a deterministic cost process $X_s := x_0 e^{\frac12 s}$, $s\in[0,T]$, for some $x_0>0$. Thus, we have $X^{t,x}_s = x e^{\frac12 (s-t)}$ for $s\in[t,T]$. 
	The smoker can (i) quit smoking at some time $s<T$ (with cost $X_s$) and die peacefully at time $T$ (with no cost), or (ii) never quit smoking (thus incurring no cost) but die painfully at time $T$ (with cost $X_T$). With hyperbolic discount function $\delta(s):=\frac{1}{1+s}$ for $s\ge 0$,  \eqref{value function} becomes  minimizing cost 
	\[
	\inf_{s\in[t,T]} \delta(s-t)X^{t,x}_{s} = \inf_{s\in[t,T]} \frac{x e^{\frac{1}{2}(s-t)}}{1+(s-t)}.
	\]
	By basic Calculus, the optimal stopping time $\widetilde{\tau}_{t,x}$ is given by
	\begin{equation}\label{ttau smoking}
	\widetilde{\tau}_{t,x} = \begin{cases} 
	t+1\quad &\hbox{if}\ t<T-1,\\
	T\quad &\hbox{if}\ t\ge T-1.
	\end{cases}
	\end{equation}
	Time inconsistency can be easily observed, and it illustrates the procrastination behavior: the smoker never quits smoking.
\end{Example}

\begin{Example}[Real Options Model]\label{eg:stopping BES(1)}
	Suppose $d=1$ and $X_s := |W_s|$, $s\ge 0$. Consider the payoff function $g(x):=x$ for $x\in\R_+$ and the hyperbolic discount function $\delta(s):=\frac{1}{1+s}$ for $s\ge 0$. The problem \eqref{value function} reduces to
	$
	v(x)=\sup_{\tau\in\Tc} \E^{x}\left[\frac{X_\tau}{1+\tau}\right].
	$
	This can be viewed as a real options problem in which the management of a large non-profitable insurance company has the intention to liquidate or sell the company, and would like to decide when to do so; see the explanations under \eqref{eq:objective function example} for details.
	
	By the argument in Pedersen and Peskir \cite{pederson2000solving}, we prove in Proposition~\ref{prop:ttau explicit} below that the optimal stopping time $\widetilde\tau_{x}$, defined in \eqref{naive} with $t=0$, has the formula
	$$\widetilde\tau_{x} = \inf\left\{s\ge 0 : X^{x}_s\ge \sqrt{1+s}\right\}.$$
	If one solves the same problem at time $t>0$ with $X_t=x\in\R_+$, the optimal stopping time is $\widetilde
	\tau_{t,x}=t+\widetilde\tau_{x} = \inf\{s\ge t : X^{t,x}_s\ge \sqrt{1+(s-t)}\}.$ The free boundary $s\mapsto \sqrt{1+(s-t)}$ is unusual in its dependence on initial time $t$. 
	From Figure~\ref{fig:free boundary}, we clearly observe time inconsistency: $\widetilde\tau_{t,x}(\omega)$ and $\widetilde\tau_{t',X^{t,x}_{t'}}(\omega)$ do not agree in general, for any $t'>t$, as they correspond to different free boundaries.
	\begin{figure}[h!]
		\centering
		\begin{tikzpicture}[scale=0.9]
		\begin{axis}[legend pos=south east, xlabel={$s$}, xmin=0,	ymin=0,
		extra x ticks={3,5},
		extra tick style={grid=major}]
		\addplot[color=blue, domain=0:15] (\x,{(1+\x)^(1/2)});
		\addlegendentry{$t=0$}
		\addplot[color=red, domain=3:15] (\x,{(1+\x-3)^(1/2)});
		\addlegendentry{$t=3$}
		\addplot[color=green, domain=5:15] (\x,{(1+\x-5)^(1/2)});
		\addlegendentry{$t=5$}
		\end{axis}
		\end{tikzpicture}
		\caption{The free boundary $s\mapsto \sqrt{1+(s-t)}$ with different initial times $t$.}
		\label{fig:free boundary}
	\end{figure}
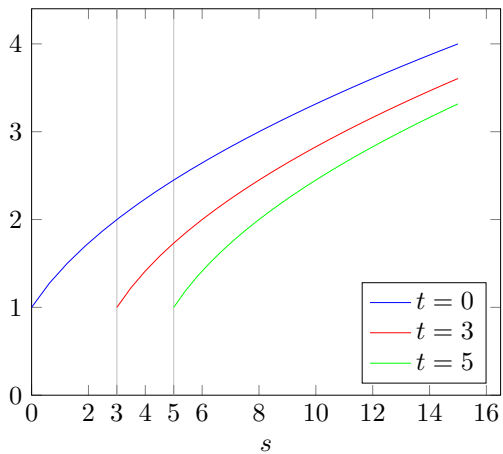
\end{Example}

As proposed in Strotz \cite{Stroz1956myopia}, to deal with time inconsistency, we need a strategy that is either {\it pre-committed} or {\it sophisticated}. A pre-committed agent finds $\widetilde\tau_{t,x}$ in \eqref{naive} at time $t$, and forces her future selves to follow $\widetilde\tau_{t,x}$ through a commitment mechanism (e.g. a contract). By contrast, a sophisticated agent works on ``consistent planning'': she anticipates the change of future preferences, and aims to find a stopping strategy that once being enforced, none of her future selves would want to deviate from it. How to precisely formulate sophisticated stopping strategies has been a challenge in continuous time, and the next section focuses on resolving this.


\section{Equilibrium Stopping Policies}\label{sec:equilibrium}

\subsection{Objective of a Sophisticated Agent}
\label{subsec:sophisticated objective}

Since one may re-evaluate and change her choice of stopping times over time, her stopping strategy is not a single stopping time, but a {\it stopping policy} defined below.

\begin{Definition}\label{def:stopping policy}
	A Borel measurable function $\tau: \X \mapsto \{0,1\}$ is called a \it{stopping policy}. We denote by $\Tc(\X)$ the set of all stopping policies.
\end{Definition}

Given current time and state $(t,x)\in\X$, a  policy $\tau\in \Tc(\X)$ governs when an agent stops: the agent stops at the first time $\tau(s,X^{t,x}_s)$ yields the value $0$, i.e. at the moment
\begin{equation}\label{L}
\begin{split}
\Lc \tau(t,x) :=\ &\inf \left\{ s\ge t ~:~ \tau(s,X^{t,x}_s)= 0 \right\}.
\end{split}
\end{equation}
To show that $\Lc \tau(t,x)$ is a well-defined stopping time, we introduce the set 
\begin{equation}\label{kernel}
\ker(\tau) := \{(t,x)\in \X : \tau(t,x) = 0\}.
\end{equation}
It is called the {\it kernel} of $\tau$, which is the collection of $(t,x)$ at which the policy $\tau$ suggests immediate stopping. Then, $\Lc\tau(t,x)$ can be expressed as
\begin{equation}\label{L'}
\Lc\tau(t,x) = \inf \left\{ s\ge t ~:~ (s,X^{t,x}_s)\in \ker(\tau)\right\}.
\end{equation}

\begin{Lemma}\label{lem:Ltau measurable} 
	For any $\tau\in\Tc(\X)$ and $(t,x)\in\X$, $\ker(\tau)\in\Bc(\X)$ and $\Lc\tau(t,x)\in \Tc_t$.
\end{Lemma}

\begin{proof}
	The Borel measurability of $\tau\in\Tc(\X)$ immediately implies $\ker(\tau)\in\Bc(\X)$. In view of \eqref{L'},
	$\Lc\tau(t,x)(\omega)= \inf \left\{ s\ge t ~:~ (s,\omega)\in E\right\}$, 
	where
	\[
	E:= \{(r,\omega)\in[t,\infty)\times\Omega : (r,X^{t,x}_r(\omega))\in \ker(\tau)\}.
	\]
	With $\ker(\tau)\in\Bc(\X)$ and the process $X^{t,x}$ being progressively measurable, $E$ is a progressively measurable set. Since the filtration $\F^t$ satisfies the usual conditions, \cite[Theorem 2.1]{Bass10} asserts that $\Lc\tau(t,x)$ is an $\F^t$-stopping time.   
	\qed\end{proof}

\begin{Remark}[Naive Stopping Policy]
	Recall the optimal stopping time $\widetilde\tau_{t,x}$ defined in \eqref{naive} for all $(t,x)\in\X$. Define $\widetilde\tau\in\Tc(\X)$ by 
	\begin{equation}\label{naive policy}
	\widetilde\tau(t,x):=
	\begin{cases}
	0,\quad \hbox{if}\ \widetilde\tau_{t,x} = t,\\
	1,\quad \hbox{if}\ \widetilde\tau_{t,x} >t.
	\end{cases}
	\end{equation}
	Note that $\widetilde\tau:\X\mapsto\{0,1\}$ is indeed Borel measurable because $\widetilde{\tau}_{t,x} =t$ if and only if
	\[
	(t,x)\in \left\{(t,x)\in\X : g(x) = \sup_{\tau\in\Tc_t} \E^{t,x}[\delta(\tau-t) g(X^{}_\tau)] \right\} \in \Bc(\X).
	\]
	Following the standard terminology (see e.g. \cite{Stroz1956myopia}, \cite{pollak1968consistent}), we call $\widetilde\tau$ the naive stopping policy as it describes the behavior of a naive agent, discussed in Subsection~\ref{subsec:time inconsistency}. 
\end{Remark}

\begin{Remark}\label{rem:dynamic optimality}
	Despite its name, the naive stopping policy $\widetilde\tau$ may readily satisfy certain optimality criterion. For example, ``dynamic optimality'' recently proposed in Pedersen and Peskir \cite{pedersen2016optimal} can be formulated in our case as follows: $\tau\in\Tc(\X)$ is dynamically optimal if there is no other $\pi\in\Tc(\X)$ such that 
	\[
	\P^{t,x}\left[J\left(\Lc\tau(t,x),X^{t,x}_{\Lc\tau(t,x)};\Lc\pi\left(\Lc\tau(t,x),X^{t,x}_{\Lc\tau(t,x)}\right)\right)>g(X^{t,x}_{\Lc\tau(t,x)})\right]>0
	\]
	for some $(t,x)\in\X$. By \eqref{naive policy} and Proposition~\ref{prop:standard result}, $\widetilde\tau$ is dynamically optimal as the above probability is always $0$. 
\end{Remark}

\begin{Example}[Real Options Model, Continued]
	\label{ex:real option 2}
	Recall the setting of Example \ref{eg:stopping BES(1)}. A naive agent follows $\widetilde\tau\in\Tc(\X)$, and the actual moment of stopping  is
	$$ 
	\Lc \widetilde\tau(t,x) = \inf\{s\ge t : \widetilde\tau(s,X^{t,x}_s) = 0\} = \inf\{s\ge t : X^{t,x}_s\ge 1\},
	$$
	which differs from the agent's original decision $\widetilde\tau_{t,x}$ in Example \ref{eg:stopping BES(1)}. \end{Example}

We can now introduce equilibrium policies. Suppose that a stopping policy $\tau\in \Tc(\X)$ is given to a sophisticated agent. At any $(t,x)\in\X$, the agent carries out the game-theoretic reasoning: ``assuming that all my future selves will follow $\tau\in\Tc(\X)$, what is the best stopping strategy at current time $t$ in response to that?'' Note that the agent at time $t$ has only two possible actions: stopping and continuation. If she stops at time $t$, she gets $g(x)$ immediately. If she continues at time $t$, 
given that all her future selves will follow $\tau\in\Tc(\X)$, she will eventually stop at the moment 
\begin{equation}\label{L*}
\begin{split}
\Lc^* \tau(t,x) :=\ &\inf \left\{ s> t ~:~ \tau(s,X^{t,x}_s) = 0 \right\}\\
=\ &\inf \left\{ s> t ~:~ (s,X^{t,x}_s)\in \ker(\tau)\right\},
\end{split}
\end{equation}
leading to the payoff
\[
J(t,x;\Lc^* \tau(t,x)) = \E^{t,x}\left[\delta(\Lc^* \tau(t,x)-t)g(X_{\Lc^* \tau(t,x)})\right].
\]
By the same argument in Lemma~\ref{lem:Ltau measurable}, $\Lc^* \tau(t,x)$ is a well-defined stopping time in $\Tc_t$. Note the subtle difference between $\Lc \tau(t,x)$ and $\Lc^* \tau(t,x)$: with the latter, the agent at time $t$ simply chooses to continue, with no regard to what $\tau\in\Tc(\X)$ suggests at time $t$. This is why we have ``$s>t$'' in \eqref{L*}, instead of ``$s\ge t$'' in \eqref{L}. 

Now, we separate the space $\X$ into three distinct regions
\begin{equation}\label{regions}
\begin{split}
S_{\tau}&:=\{(t,x)\in \X~:~g(x)>  J(t,x; \Lc^*\tau(t,x)) \},\\
C_{\tau}&:=\{(t,x)\in \X~:~g(x)<  J(t,x; \Lc^*\tau(t,x)) \},\\
I_{\tau}&:=\{(t,x)\in \X~:~g(x)=  J(t,x; \Lc^*\tau(t,x)) \}.
\end{split}
\end{equation}
Some conclusions can be drawn: 
\begin{itemize}
	\item [1.] If $(t,x)\in S_\tau$, the agent should stop immediately at time $t$. 
	\item [2.] If $(t,x)\in C_{\tau}$, the agent should continue at time $t$. 
	\item [3.] If $(t,x)\in I_{\tau}$, the agent is indifferent between stopping and continuation at current time; there is then no incentive for the agent to deviate from the originally assigned stopping strategy $\tau(t,x)$.
\end{itemize}
To summarize, for any $(t,x)\in\X$, the best stopping strategy at current time (in response to future selves following $\tau\in\Tc(\X)$) is 
\begin{equation}\label{Theta}
\Theta \tau(t,x) := \begin{cases}
0 & \text{ for }(t,x)\in S_\tau\\
1 & \text{ for }(t,x)\in C_\tau\\
\tau(t,x) & \text{ for }(t,x)\in I_\tau.
\end{cases}.
\end{equation}
The next result shows that $\Theta\tau:\X\mapsto\{0,1\}$ is again a stopping policy.

\begin{Lemma}\label{lem:Thetatau measurable}
	For any $\tau\in\Tc(\X)$, $S_\tau$, $C_\tau$, and $I_\tau$ belong to $\Bc(\X)$, and $\Theta\tau\in\Tc(\X)$. 
\end{Lemma}

\begin{proof}
	Since $\Lc^*\tau(t,x)$ is the first hitting time to the Borel set $\ker(\tau)$, the map $(t,x)\mapsto J(t,x;\Lc^*\tau(t,x)) = \E^{t,x}[\delta(\Lc^*\tau(t,x)-t) g(X^{}_{\Lc^*\tau(t,x)})]$ is Borel measurable, and thus $S_\tau$, $I_\tau$, and $C_\tau$ all belong to $\Bc(\X)$. Now, by \eqref{Theta},
	$
	\ker(\Theta\tau) = S_\tau \cup (I_\tau\cap\ker(\tau))\in\Bc(\X),
	$
	which implies that $\Theta\tau\in\Tc(\X)$. 
	\qed\end{proof}

By Lemma \ref{lem:Thetatau measurable}, $\Theta$ can be viewed as an operator acting on the space $\Tc(\X)$. For any initial $\tau\in \Tc(\X)$, $\Theta: \Tc(\X)\mapsto \Tc(\X)$ generates
a new policy $\Theta\tau\in \Tc(\X)$. The switch from $\tau$ to $\Theta\tau$ corresponds to an additional level of  strategic reasoning in Game Theory, as discussed below Corollary \ref{coro:naive satisfies}.

\begin{Definition} [Equilibrium Stopping Policies] \label{def:equilibrium}
	We say $\tau\in\Tc(\X)$ is an equilibrium stopping policy if $\Theta \tau(t,x) = \tau(t,x)$ for all $(t,x)\in\X$.  
	We denote by $\Ec(\X)$ the collection of all equilibrium stopping policies.
\end{Definition}

The term ``equilibrium'' is used as a connection to subgame-perfect Nash equilibria in an inter-temporal game among current self and future selves. This equilibrium idea was invoked in stochastic control under time inconsistency; see e.g. \cite{EL06}, \cite{Ekeland2008investment}, \cite{Ekeland2012time}, and \cite{Bjork2014theory-discrete}. A contrast with the stochastic control literature needs to be pointed out.

\begin{Remark}[Comparison with Stochastic Control]\label{rem:comparison to control}
	In time-inconsistent stochastic control, local perturbation of strategies on small time intervals $[t,t+\eps]$ is the standard way to define equilibrium controls. In our case, local perturbation is carried out instantaneously at time $t$. This is because an instantaneously-modified stopping strategy may already change the expected discounted payoff significantly, whereas a control perturbed only at time $t$ yields no effect.
\end{Remark}

The first question concerning Definition~\ref{def:equilibrium} is the existence of an equilibrium stopping policy. 
Finding at least one such a policy turns out to be easy.

\begin{Remark} [Trivial Equilibrium]\label{rem:trivial}
	Define $\tau\in\Tc(X)$ by $\tau(t,x):= 0$ for all $(t,x)\in\X$. Then $\Lc\tau(t,x)=\Lc^*\tau(t,x)=t$, and thus $J(t,x;\Lc^*\tau(t,x)) = g(x)$ for all $(t,x)\in\X$. This implies $I_\tau = \X$. We then conclude from \eqref{Theta} that $\Theta\tau(t,x)  = \tau(t,x)$ for all $(t,x)\in\X$, which shows $\tau\in\Ec(\X)$. We call it the trivial equilibrium stopping policy.  
\end{Remark}

\begin{Example}[Smoking Cessation, Continued]\label{eg:smoking continued}
	Recall the setting in Example~\ref{eg:smoking}. 
	Observe from \eqref{ttau smoking} and \eqref{naive policy} that $\Lc^*\widetilde\tau(t,x) = T$ for all $(t,x)\in\X$. Then, 
	\[
	\delta(\Lc^*\widetilde\tau(t,x)-t) X^{t,x}_{\Lc^*\widetilde\tau(t,x)} = \frac{X^{t,x}_T}{1+T-t} = \frac{xe^{\frac{1}{2}(T-t)}}{1+T-t}. 
	\]
	Since $e^{\frac12 s}=1+s$ has two solutions $s=0$ and $s=s^*\approx 2.51286$, and $e^{\frac12 s}> 1+s$ iff $s>s^*$, the above equation implies
	$S_{\widetilde\tau}=\{(t,x):t <T-s^*\}$, $C_{\widetilde\tau} = \{(t,x):t\in(T-s^*,T)\}$, and $I_{\widetilde\tau} = \{(t,x):t = T-s^*\ \hbox{or}\ T\}$. We therefore get 
	\[
	\Theta\widetilde\tau(t,x) = 
	\begin{cases}
	0\quad &\hbox{for}\ t<T-s^*,\\
	1\quad &\hbox{for}\ t\ge T-s^*.
	\end{cases}
	\]
	Whereas a naive smoker delays quitting smoking indefinitely (as in Example~\ref{eg:smoking}), the first level of strategic reasoning (i.e. applying $\Theta$ to $\widetilde\tau$ once) recognizes this procrastination behavior and pushes the smoker to quit immediately, unless he is already too old (i.e. $t\ge T-s^*$). It can be checked that $\Theta\widetilde\tau$ is already an equilibrium, i.e. $\Theta^2\widetilde\tau(t,x)= \Theta\widetilde\tau(t,x)$ for all $(t,x)\in\X$.
\end{Example}

It is worth noting that in the classical case of exponential discounting, characterized by \eqref{identity}, the naive stopping policy $\widetilde\tau$ in \eqref{naive policy} is already an equilibrium. 

\begin{Proposition}\label{prop:ttau in Ec}
	Under \eqref{identity}, $\widetilde\tau\in\Tc(\X)$ defined in \eqref{naive policy} belongs to $\Ec(\X)$.
\end{Proposition}

\begin{proof}
	The proof is delegated to Appendix~\ref{subsec:ttau in Ec}.
	\qed\end{proof}

\subsection{The Main Result}
\label{subsec:finding equilibrium}

In this subsection, we look for equilibrium policies through fixed-point iterations. For any $\tau\in\Tc(\X)$, we apply $\Theta$ to $\tau$ repetitively  until we reach an equilibrium policy. 
In short, we define $\tau_0$ by
\begin{equation}\label{tau0}
\tau_0(t,x) := \lim_{n\to\infty} \Theta^n\tau(t,x)\quad \forall (t,x)\in\X,
\end{equation}
and take it as a candidate equilibrium policy. To make this argument rigorous, we need to show (i) the limit in \eqref{tau0} converges, so that $\tau_0$ is well-defined; (ii) $\tau_0$ is indeed an equilibrium policy, i.e. $\Theta\tau_0 = \tau_0$.  
To this end, we impose the condition:

\begin{Assumption}\label{asm:DI}
	The function $\delta$ satisfies 
	$
	\delta(s)\delta(t)\le \delta(s+t)
	$ for all $s, t\ge 0$.
\end{Assumption}
Assumption~\ref{asm:DI} is closely related to {\it decreasing impatience ($DI$)} in Behavioral Economics. It is well-documented in empirical studies, e.g. \cite{Thaler81}, \cite{LT89}, \cite{LP92}, that people admits $DI$: when choosing between two rewards, people are more willing to wait for the larger reward (more patient) when these two rewards are further away in time. For instance, in the two scenarios (i) getting \$100 today or \$110 tomorrow, and (ii) getting \$100 in 100 days or \$110 in 101 days, 
people tend to choose \$100 in (i), but \$110 in (ii). 

Following \cite[Definition 1]{Prelec04} and \cite{noor2009a}, \cite{noor2009b}, $DI$ can be formulated under current context as follows: 
the discount function $\delta$ induces $DI$ if 
\begin{equation}\label{DI}
\hbox{for any } s\ge 0,\quad t\mapsto\frac{\delta(t+s)}{\delta(t)}\quad \hbox{is strictly increasing}. 
\end{equation}
Observe that \eqref{DI} readily implies Assumption~\ref{asm:DI}, as $\delta(t+s)/\delta(t)\ge \delta(s)/\delta(0) = \delta(s)$ for all $s,t\ge 0$. That is, Assumption  \ref{asm:DI} is automatically true under $DI$. Note that Assumption \ref{asm:DI} is more general than $DI$, as it obviously includes the classical case of exponential discounting, characterized by \eqref{identity}. 

The main convergence result for \eqref{tau0} is the following:

\begin{Proposition}\label{thm:main 1}
	Let Assumption~\ref{asm:DI} hold. If $\tau\in\Tc(\X)$ satisfies 
	\begin{equation}\label{1st iteration decreasing}
	\ker(\tau)\subseteq \ker(\Theta\tau),
	\end{equation}
	then 
	\begin{equation}\label{ker increasing}
	\ker(\Theta^n \tau)\subseteq  \ker(\Theta^{n+1} \tau),\quad \forall n\in\N.
	\end{equation}
	Hence, $\tau_0$ in \eqref{tau0} is a well-defined element in $\Tc(\X)$, with $\ker(\tau_0)=\bigcup_{n\in\N}\ker(\Theta^{n}\tau)$.
\end{Proposition}

\begin{proof}
	The proof is delegated to Appendix~\ref{subsec:LTheta^n decreasing}.
	\qed\end{proof}

Condition \eqref{1st iteration decreasing} means that at any $(t,x)\in\X$ where the initial policy $\tau$ indicates immediate stopping, the new policy $\Theta\tau$ agrees with it; however, it is possible that at some $(t,x)\in\X$ where $\tau$ indicates continuation, $\Theta\tau$ suggests immediate stopping, based on the game-theoretic reasoning in Subsection~\ref{subsec:sophisticated objective}.
Note that \eqref{1st iteration decreasing} is not very restrictive, as it already covers all hitting times to subsets of $\X$ that are open (or more generally, half-open in $[0,\infty)$ and open in $\R^d$), as explained below.

\begin{Remark} 
	Let $E$ be a subset of $\X$ that is ``open'' in the sense that for any $(t,x)\in E$, there exists $\eps>0$ such that $(t,x)\in[t, t+\eps)\times B_\eps(x)\subseteq E$, where $B_\eps(x):=\{y\in\R^d : |y-x|<\eps\}$. Define $\tau\in\Tc(\X)$ by $\tau(t,x)=0$ if and only if $(t,x)\in E$. 
	Since $\ker(\tau) = E$ is ``open'', for any $(t,x)\in\ker(\tau)$, we have $\Lc^*\tau(t,x) = t$, which implies $(t,x)\in I_\tau$. Thus, $\ker(\tau)\subseteq I_\tau$. It follows that \eqref{1st iteration decreasing} holds, as $\ker(\tau) \subseteq S_\tau\cup\ker(\tau) = S_\tau\cup (I_\tau\cap\ker(\tau)) = \ker(\Theta\tau)$, where the last equality is due to \eqref{Theta}.
	
	The stopping policy $\tau$ corresponds to the stopping times $T_{t,x}:=\inf\{s\ge t : (s,X_s^{t,x})\in E\}$ for all $(t,x)\in\X$. In particular, if $E = [0,\infty)\times F$ where $F$ is an open set in $\R^d$, the corresponding stopping times are $T'_{t,x}:=\inf\{s\ge t : X_s^{t,x}\in F\}$, $(t,x)\in\X$.
\end{Remark}

Moreover, the naive stopping policy $\widetilde\tau$ also satisfies \eqref{1st iteration decreasing}. 

\begin{Proposition}\label{prop:naive satisfies}
	$\widetilde\tau\in\Tc(\X)$ defined in \eqref{naive policy} satisfies \eqref{1st iteration decreasing}. 
\end{Proposition}

\begin{proof}
	The proof is delegated to Appendix~\ref{subsec:naive satisfies}.
	\qed\end{proof}

The next theorem is the main result of our paper. It shows that the fixed-point iteration in \eqref{tau0} indeed converges to an equilibrium policy. 

\begin{Theorem}\label{thm:main 2}
	Let Assumption~\ref{asm:DI} hold. If $\tau\in\Tc(\X)$ satisfies \eqref{1st iteration decreasing}, then 
	$\tau_0$ defined in \eqref{tau0} belongs to $\Ec(\X)$.
\end{Theorem}

\begin{proof}
	The proof is delegated to Section~\ref{subsec:proof of lem:technical}. 
	\qed\end{proof}

The following result for the naive stopping policy $\widetilde\tau$, defined in \eqref{naive policy}, is a direct consequence of Proposition~\ref{prop:naive satisfies} and Theorem~\ref{thm:main 2}. 

\begin{Corollary}\label{coro:naive satisfies}
	Let Assumption~\ref{asm:DI} hold. The stopping policy $\widetilde\tau_0\in\Tc(\X)$ defined by
	\begin{equation}\label{ttau0}
	\widetilde\tau_0(t,x) := \lim_{n\to\infty} \Theta^n\widetilde\tau(t,x)\quad \forall (t,x)\in\X
	\end{equation}
	belongs to $\Ec(\X)$. 
\end{Corollary}

Our iterative approach, as in \eqref{tau0}, contributes to the literature of time inconsistency in two ways. First, the standard approach for finding equilibrium strategies in continuous time is solving a system of non-linear equations (the so-called extended HJB equation), as proposed in \cite{Ekeland2008investment} and \cite{BKM16}. Solving this system of equations is difficult; 
and even when it is solved (as in the special cases in \cite{Ekeland2008investment} and \cite{BKM16}), 
we just obtain one particular equilibrium, and it is unclear how other equilibrium strategies can be found. 
Our iterative approach provides a potential remedy here. 
We can find different equilibria simply by starting the iteration \eqref{tau0} with different initial policies $\tau\in\Tc(\X)$. In some cases, we are able to find {\it all} equilibria, and obtain a complete characterization of $\Ec(\X)$; see Proposition~\ref{prop:entire Ec} below. 

Second, while the continuous-time formulation of equilibrium strategies was initiated in \cite{EL06}, the ``origin'' of an equilibrium strategy has not been addressed. This question is important as people do not start with using  equilibrium strategies. People have their own initial strategies, determined by a variety of factors such as classical optimal stopping theory, personal habits, and popular rules of thumb in the market. Once an agent starts to do game-theoretic reasoning and look for equilibrium strategies, she is not satisfied with an arbitrary equilibrium. Instead, she works on improving her initial strategy to turn it into an equilibrium. This improving process is absent from \cite{EL06}, \cite{Ekeland2008investment}, and \cite{BKM16}, but it is in fact well-known in Game Theory as the hierarchy of strategic reasoning in \cite{Stahl93} and \cite{SW94}.
Our iterative approach embodies this 
framework: given an initial $\tau\in\Tc(\X)$, $\Theta^n\tau\in\Tc(\X)$ corresponds to level-$n$ strategic reasoning in \cite{SW94}, and $\tau_0:= \lim_{n\to\infty}\Theta^n\tau$ reflects full rationality of ``smart$_\infty$'' players in \cite{Stahl93}. Hence, our formulation complements the literature of time inconsistency in that it not only defines what an equilibrium is, but explains where an equilibrium is coming from. This in turn provides ``agent-specific'' results: it assigns one specific equilibrium to each agent according to her initial behavior. 

In particular, Corollary~\ref{coro:naive satisfies} 
specifies the connection between the naive behavior and the sophisticated one. While these behaviors have been widely discussed in the literature, their relation has not been stated mathematically as precisely as in \eqref{ttau0}.  



\subsection{The Time-Homogeneous Case}
Suppose the state process $X$ is time-homogeneous, i.e. $X_s=f(W_s)$ for some measurable $f:\R^d\mapsto\R$; or, the coefficients $b$ and $\sigma$ in \eqref{SDE} does not depend on $t$. 
The objective function \eqref{objective function} then reduces to $J(x;\tau) := \E^x[\delta(\tau) g(X^{}_\tau)]$ for $x\in\R^d$ and $\tau\in\Tc$,
where the superscript of $\E^x$ means $X_0=x$. The decision to stop or to continue then depends on the current state $x$ only. The formulation in Subsection~\ref{subsec:sophisticated objective} reduces to:

\begin{Definition}\label{def:homogeneous case}
	When $X$ is time-homogeneous, a Borel measurable $\tau: \R^d\mapsto \{0, 1\}$ is called a stopping policy, and we denote by $\Tc(\R^d)$ the set of all stopping policies. Given $\tau\in\Tc(\R^d)$ and $x\in\R^d$, we define, similarly to \eqref{kernel}, \eqref{L}, and \eqref{L*}, $\ker(\tau):= \{x\in\R^d : \tau(x)=0\}$, $\Lc\tau(x):=\inf\{t\ge 0 : \tau(X^{x}_t)=0\}$, and $\Lc^*\tau(x):=\inf\{t> 0 : \tau(X^{x}_t)=0\}$.
	Furthermore, we say $\tau\in\Tc(\R^d)$ is an equilibrium stopping policy if $\Theta\tau(x) =\tau(x) $ for all $x\in\R^d$, where 
	\begin{equation}\label{Theta'}
	\Theta \tau(x) := \begin{cases}
	0 & \text{ if  }\ x\in S_\tau:= \{x : g(x)> \E^x[\delta(\Lc^*\tau(x)) g(X^{}_{\Lc^*\tau(x)})]\},\\
	1 & \text{ if  }\ x\in C_\tau :=\{x: g(x)< \E^x[\delta(\Lc^*\tau(x)) g(X^{}_{\Lc^*\tau(x)})]\},\\
	\tau(x) & \text{ if }\ x\in I_\tau := \{ x : g(x)= \E^x[\delta(\Lc^*\tau(x)) g(X^{}_{\Lc^*\tau(x)})]\}.
	\end{cases}
	\end{equation}
\end{Definition}

\begin{Remark}\label{rem:homogeneous case}
	When $X$ is time-homogeneous, all the results in Subsection~\ref{subsec:finding equilibrium} hold, with $\Tc(\X)$, $\Ec(\X)$, $\ker(\tau)$, and $\Theta$ replaced by the corresponding ones in Definition~\ref{def:homogeneous case}. 
	Proofs of these statements are similar to, and in fact easier than, those in Subsection~\ref{subsec:finding equilibrium}, thanks to the homogeneity in time.
\end{Remark}


\section{A Detailed Case Study: Stopping of BES(1)} \label{sec:examples}

In this section, we recall the setup of Example~\ref{eg:stopping BES(1)}, with hyperbolic discount function
\begin{equation}\label{hyperbolic}
\delta(s):=\frac{1}{1+\beta s} \quad \forall s\ge 0,
\end{equation}
where $\beta>0$ is a fixed parameter. The state process $X$ is a one-dimensional Bessel process, i.e. $X_t=|W_t|$, $t\ge 0$, where $W$ is a one-dimensional Brownian motion. With $X$ being time-homogeneous, we will follow Definition~\ref{def:homogeneous case} and Remark~\ref{rem:homogeneous case}. Also, the classical optimal stopping problem \eqref{value function} reduces to 
\begin{equation}\label{eq:objective function example}
v(x) = \sup_{\tau\in\Tc}\E^{x}\left[\frac{X^{}_\tau}{1+\beta \tau}\right]\quad \hbox{for}\ x\in\R_+.
\end{equation}
This can be viewed as a real options problem, as explained below.

By \cite{TM03} and the references therein, when the surplus (or reserve) of an insurance company is much larger than the size of each individual claim, the dynamics of the surplus process can be approximated by $dR_t = \mu dt + \sigma dW_t$ with $\mu = p-\E[Z]$ {and} $\sigma =  \sqrt{\E[Z^2]}$.
Here, $p>0$ is the premium rate, and $Z$ is a random variable that represents the size of each claim. Suppose that an insurance company is non-profitable with $\mu =0$, i.e. it uses all the premiums collected to cover incoming claims. Also assume that the company is large enough to be considered ``systemically important'', so that when its surplus hits zero, the government will provide monetary support to bring it back to positivity, as in the recent financial crisis. The dynamics of $R$ is then a Brownian motion reflected at the origin. Thus, \eqref{eq:objective function example} describes a real options problem in which the management of a large non-profitable insurance company has the intention to liquidate or sell the company, and would like to decide when to do so. 

An unusual feature of \eqref{eq:objective function example} is that the discounted process $\{\delta(s)v(X_s^{x})\}_{s\ge 0}$ may not be a supermartingale. 
This makes solving \eqref{eq:objective function example} for the optimal stopping time $\widetilde\tau_x$, 
defined in \eqref{naive} with $t=0$, nontrivial. As shown in Appendix~\ref{subsec:ttau explicit}, we need an auxiliary value function, and use the method of time-change in \cite{pederson2000solving}. 

\begin{Proposition}\label{prop:ttau explicit}
	For any $x\in\R_+$, the optimal stopping time $\widetilde\tau_{x}$ of \eqref{eq:objective function example} (defined in \eqref{naive} with $t=0$) admits the explicit formula
	\begin{equation}\label{ttau explicit'}
	\widetilde\tau_{x} = \inf \left\{ s\ge 0 :  X^{x}_s\geq \sqrt{1/\beta+s} \right\}.
	\end{equation}
	Hence, the naive stopping policy $\widetilde\tau\in\Tc(\R_+)$, defined in \eqref{naive policy}, is given by
	\begin{equation}\label{ttau explicit}
	\widetilde\tau(x) := \indi_{[0,\sqrt{1/\beta})}(x) \quad \forall x\in\R_+.
	\end{equation}
\end{Proposition}

\begin{proof}
	The proof is delegated to Appendix~\ref{subsec:ttau explicit}.
	\qed\end{proof}

\subsection{Characterization of equilibrium policies}
\label{subsec:characterization}

\begin{Lemma}\label{lem:closure ker}
	For any $\tau\in\Tc(\R_+)$, consider $\tau'\in\Tc(\R_+)$ with $\ker(\tau') := \overline{\ker(\tau)}$. Then $\Lc^*\tau(x) = \Lc\tau(x)=\Lc\tau'(x)=\Lc^*\tau'(x)$ for all $x\in\R_+$. Hence, $\tau\in\Ec(\R_+)$ if and only if $\tau'\in\Ec(\R_+)$. 
\end{Lemma}

\begin{proof}
	If $x\in\R_+$ is in the interior of $\ker(\tau)$, $\Lc^*\tau(x) = \Lc\tau(x)=0=\Lc\tau'(x)=\Lc^*\tau'(x)$. Since a one-dimensional Brownian motion $W$ is monotone in no interval, if $x\in \ker(\tau')\setminus \ker(\tau)$, $\Lc^*\tau(x) = \Lc\tau(x)=0=\Lc\tau'(x)=\Lc^*\tau'(x)$; if $x\notin\ker(\tau')$, then $\Lc^*\tau(x) = \Lc\tau(x)=\inf\{s\ge 0: |W^x|\in\ker(\tau)\}= \inf\{s\ge 0: |W^x|\in\overline{\ker(\tau)}\}=\Lc\tau'(x)=\Lc^*\tau'(x)$. Finally, we deduce from \eqref{Theta'} and $\Lc^*\tau(x) = \Lc^*\tau'(x)$ for all $x\in\R_+$ that $\tau\in\Ec(\R_+)$ implies $\tau'\in\Ec(\R_+)$, and vice versa.
	\qed\end{proof}

The next result shows that every equilibrium policy corresponds to the hitting time to a certain threshold. Recall that a set $E\subset \R_+$ is called totally disconnected if the only nonempty connected subsets of $E$ are singletons, i.e. $E$ contains no interval.

\begin{Lemma} \label{lem:tau_a}
	For any $\tau\in\Ec(\R_+)$, define $a := \inf\left( \ker(\tau)\right)\ge 0$. Then, the Borel set 
	$E:=\{x\ge a:x\notin \ker(\tau)\}$ is totally disconnected. Hence, $\overline{\ker(\tau)}=[a,\infty)$ and the stopping policy $\tau_a$, defined by $\tau_a(x):= \indi_{[0,a)}(x)$ for $x\in\R_+$, belongs to $\Ec(\R_+)$.
\end{Lemma}

\begin{proof}
	The proof is delegated to Appendix \ref{app:threshold}
	\qed\end{proof}

The converse question is for which $a\ge 0$ the policy $\tau_a\in\Tc(\R)$ is an equilibrium. To answer this, we need to find the sets $S_{\tau_a}$, $C_{\tau_a}$, and $I_{\tau_a}$ in \eqref{Theta'}. By Definition~\ref{def:homogeneous case}, 
\begin{equation}\label{L and L* tau_a}
\Lc\tau_a(x) = T^x_a:=\inf\{s\ge 0 : X^{x}_s\ge a\},\quad \Lc^*\tau_a(x) = \inf\{s> 0 : X^{x}_s\ge a\}.
\end{equation} 
Note that  $\Lc\tau_a(x) = \Lc^*\tau_a(x)$, by an argument similar to the proof of Lemma~\ref{lem:closure ker}. 
As a result, for $x\ge a$, we have  $J(x;\Lc^*\tau_a(x)) = J(x;0) =x$, which implies
\begin{equation}\label{in I}
[a,\infty) \subseteq I_{\tau_a}.
\end{equation}
For $x\in[0,a)$, we need the lemma below, whose proof is delegated to Appendix~\ref{subsec:lem eta^a}.

\begin{Lemma}\label{lem:eta^a}
	Recall $T^x_a$ in \eqref{L and L* tau_a}. {On the space $\{(x,a)\in\R^2_+ ~:~ a\ge x\}$, define  
		\[
		\eta(x, a) := \E^{x}\left[\frac{a}{1+\beta T^x_a}\right].
		\]}
	\begin{itemize}
		\item [(i)] For any $a\ge 0$, $x\mapsto\eta(x,a)$ is strictly increasing and strictly convex on [0,a], and satisfies $0<\eta(0,a)<a$ and $\eta(a,a)=a$.
		\item [(ii)] {For any $x\ge 0$, $\eta(x,a)\to 0$ as $a\to\infty$.}
		\item [(iii)] There exists a unique {$a^*\in (0,1/\sqrt{\beta})$} such that for any $a>a^*$, there is a unique solution {$x^*(a)\in (0,a^*)$}
		of $\eta(x,a)=x$. Hence, $\eta(x,a)>x$ for $x< x^*(a)$ and $\eta(x,a)<x$ for $x> x^*(a)$. On the other hand, $a\le a^*$ implies that $\eta(x,a)>x$ for all $x\in (0,a)$.
	\end{itemize}
	The figure below illustrates $x\mapsto \eta(x,a)$ under different scenarios $a\le a^*$ and $a>a^*$.
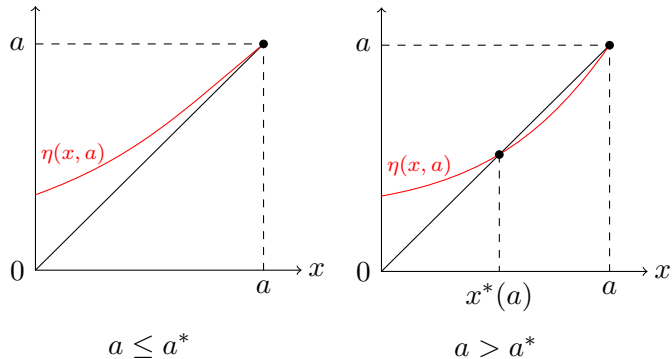
	\begin{center}
		\begin{tikzpicture}
		\draw [<->] (3.5,0) -- (0,0) node [left] {$0$} -- (0,3.5);
		\draw (0,0) -- (3,3);
		\draw[color=red] (0,1) to [out=20,in=-140] (3,3);
		\node at (0.5,1.55) {\scriptsize{{\color{red}$\eta(x,a)$}}};
		\draw [dashed] (0,3) node [left] {$a$} -- (3,3)
		-- (3,0) node (a1) [below] {$a$};
		\node at (3.7,0) {$x$};
		\node (<a*) at (1.5,-1) {$a\le a^*$};
		\draw [fill] (3,3) circle [radius=.05];
		\end{tikzpicture}
		\begin{tikzpicture}
		\draw [<->] (3.5,0) -- (0,0) node [left] {$0$} -- (0,3.5);
		\draw (0,0) -- (3,3);
		\draw[color=red] (0,1) to [out=10,in=-125] (3,3);
		\node at (0.5,1.4) {\scriptsize{{\color{red}$\eta(x,a)$}}};
		\draw [dashed] (0,3) node [left] {$a$} -- (3,3)
		-- (3,0) node (a2) [below] {$a$};
		\draw [dashed] (1.55,1.55) -- (1.55,0) node [below] {$x^*(a)$};
		\node at (3.7,0) {$x$};
		\node (>a*) at (1.5,-1) {$a> a^*$};
		\draw [fill] (1.55,1.55) circle [radius=.05];
		\draw [fill] (3,3) circle [radius=.05];
		\end{tikzpicture}
	\end{center}
	\vspace{-0.3cm}
\end{Lemma}

\noindent
We now separate the case $x\in[0,a)$ into two sub-cases:
\begin{itemize}
	\item [1.] If $a\le a^*$, Lemma~\ref{lem:eta^a} (iii) shows that $J(x;\Lc^*\tau_a(x)) =\eta(x,a) > x$, and thus
	\begin{equation}\label{in C}
	[0,a)\subseteq C_{\tau_a}.
	\end{equation}
	\item [2.] If $a> a^*$, then by Lemma~\ref{lem:eta^a} (iii),
	\begin{equation}\label{in S,C,I}
	J(x;\Lc^*\tau_a(x)) = \eta(x,a)
	\begin{cases}
	>x,\quad &\hbox{if}\ x\in[0, x^*(a)), \\
	=x,\quad &\hbox{if}\ x = x^*(a), \\
	<x,\quad &\hbox{if}\ x\in (x^*(a),a).
	\end{cases}
	\end{equation}
\end{itemize} 
By \eqref{in I}, \eqref{in C}, \eqref{in S,C,I}, and the definition of $\Theta$ in \eqref{Theta'},
\begin{align}
&\hbox{if}\ a\le a^*,\quad  \Theta\tau_a(x) = \indi_{[0,a)}(x)+\tau_a(x) \indi_{[a,\infty)}(x)\equiv\tau_a(x);\nonumber\\
&\hbox{if}\ a> a^*,\quad \Theta\tau_a(x) =  \indi_{[0,x^*(a))}(x)+ \tau_a(x)\mathbf{1}_{\{x^*(a)\}\cup [a,\infty)}(x) \not\equiv \tau_a(x). \label{Thetatau_a 2}
\end{align}

\begin{Proposition}\label{prop:entire Ec}
	$\tau_a$ defined in Lemma~\ref{lem:tau_a} belongs to $\Ec(\R_+)$ if and only if $a\in[0,a^*]$, where $a^*>0$ is characterized by  
	$a^* \int_0^\infty e^{-s} \sqrt{2\beta s}\tanh(a^* \sqrt{2\beta s}) ds =1.$ 
	Moreover, 
	\begin{equation}\label{Ec=}
	\Ec(\R_+) = \{\tau\in \Tc(\R_+) ~:~ \overline{\ker(\tau)} = [a,\infty)\ \hbox{for some}\ a \in [0,a^*]\}.
	\end{equation}
\end{Proposition}

\begin{proof}
	The derivation of ``$\tau_a\in\Ec(\R_+)\iff a\in[0,a^*]$'' is presented in the discussion above the proposition. By the proof of Lemma~\ref{lem:eta^a} in Appendix~\ref{subsec:lem eta^a}, $a^*$ satisfies $\eta_a(a^*,a^*)=1$, which leads to the characterization of $a^*$. Now, for any $\tau\in\Tc(\R_+)$ with $\overline{\ker(\tau)} = [a,\infty)$ and $a \in [0,a^*]$, Lemma~\ref{lem:closure ker} implies $\tau\in\Ec(\R_+)$. For any $\tau\in\Ec(\R_+)$, set $a:=\inf(\ker(\tau))$. By Lemma~\ref{lem:tau_a}, $\overline{\ker(\tau)} = [a,\infty)$ and $\tau_a\in\Ec(\R_+)$. The latter implies $a \in [0,a^*]$ and thus completes the proof.
	\qed\end{proof}

\begin{Remark}[Estimating $a^*$]\label{rem:solve a^*}
	With $\beta=1$, numerical computation gives $a^* \approx 0.946475$. It follows that for a general $\beta>0$, $a^*\approx 0.946475/\sqrt{\beta}$.
\end{Remark}

For $a>a^*$, although $\tau_a\notin\Ec(\R_+)$ by Proposition~\ref{prop:entire Ec}, we may use the iteration in \eqref{tau0} to find a stopping policy in $\Ec(\R_+)$. Here, the repetitive application of $\Theta$ to $\tau_a$ has a simple structure: to reach an equilibrium, we need only {\it one} iteration. 


\begin{Remark}\label{rem:one iteration}
	Fix $a>a^*$, and recall $x^*(a)\in(0,a^*)$ in Lemma~\ref{lem:eta^a} (iii). By \eqref{Thetatau_a 2}, 
	\[
	\Theta\tau_a(x) = \tau'_{x^*(a)}(x) := \indi_{[0,x^*(a)]}(x)\quad \hbox{for all}\ x\in\R_+.
	\]
	Equivalently, $\ker(\Theta\tau_a)=\ker(\tau'_{x^*(a)})=(x^*(x),\infty)$. 
	Since $\overline{\ker(\tau'_{x^*(a)})}=[x^*(a),\infty)$ and $x^*(a)\in(0,a^*)$, we conclude from \eqref{Ec=} that $\tau'_{x^*(a)}\in\Ec(\R_+)$.   
\end{Remark}



Recall \eqref{ttau0} which connects the naive and sophisticated behaviors. With the naive strategy $\widetilde\tau\in\Tc(\R_+)$ given explicitly in \eqref{ttau explicit}, Proposition~\ref{prop:entire Ec} and Remark~\ref{rem:solve a^*} imply $\widetilde\tau\notin\Ec(\R_+)$. We may find the corresponding equilibrium as in Remark~\ref{rem:one iteration}.  

\begin{Remark}\label{rem:BES(1) sophisticated}
	Set $\widetilde{a} :=1/\sqrt{\beta}$. By \eqref{ttau explicit} and Remark~\ref{rem:one iteration}, $\Theta\widetilde\tau = \Theta\tau_{\widetilde{a}} = \tau'_{x^*(\widetilde{a})}\in\Ec(\R_+)$. In view of the proof of Lemma~\ref{lem:eta^a} in Appendix~\ref{subsec:lem eta^a}, we can find $x^*(\widetilde{a})$ by solving $\eta(1/\sqrt{\beta},x)=x$, i.e. $\frac{1}{\sqrt{\beta}} \int_0^\infty e^{-s} \cosh(x\sqrt{2\beta s})\sech(\sqrt{2s}) ds =x$, for $x$. Numerical computation shows $x^*(\widetilde{a}) \approx 0.92195/\sqrt{\beta}$, and thus $x^*(\widetilde a)<a^*$ by Remark~\ref{rem:solve a^*}. This verifies $\tau'_{x^*(\widetilde{a})}\in\Ec(\R_+)$, thanks to \eqref{Ec=}.
\end{Remark}

\begin{Remark}\label{rem:PP and ours}
	Recall ``static optimality'' and ``dynamic optimality'' in Remarks~\ref{rem:static optimality} and \ref{rem:dynamic optimality}. By Proposition~\ref{prop:ttau explicit}, $\widetilde\tau_x$ in \eqref{ttau explicit'} is statically optimal for  $x\in\R_+$ fixed, while $\widetilde\tau$ in \eqref{ttau explicit} is dynamically optimal. This is reminiscent of the situation in Theorem 3 of \cite{pedersen2016optimal}. Moreover, $\tau\in\Tc(\R_+)$ defined by $\tau(x):= \indi_{[0,b)}(x)$, $x\in\R_+$, is dynamically optimal for all $b\ge \sqrt{1/\beta}$, thanks again to  Proposition~\ref{prop:ttau explicit}.
\end{Remark}



\subsection{Further consideration on selecting equilibrium policies}
\label{sec:new class}
\label{subsec:optimal consistent}

In view of \eqref{Ec=}, it is natural to ask which equilibrium in $\Ec(\R_+)$ one should employ. According to standard Game Theory literature discussed below Corollary~\ref{coro:naive satisfies}, a sophisticated agent should employ the specific equilibrium generated by her initial stopping policy $\tau$, through the iteration \eqref{tau0}. Now, imagine that an agent is ``born'' sophisticated: 
she does not have any previously-determined initial stopping policy, and intends to apply an equilibrium policy straight away. A potential way to formulate her stopping problem is the following: 
\begin{equation}\label{eq:problem2}
\sup_{\tau\in\Ec(\R_+)} J(x;\Lc\tau(x)) = \sup_{a\in[0,a^*]}J(x;\Lc\tau_a(x))= \sup_{a\in [x,a^*\vee x]}\E^{x}\left[\frac{a}{1+\beta T^x_a}\right]. 
\end{equation}
where the first equality follows from Proposition~\ref{prop:entire Ec} and Lemma~\ref{lem:closure ker}.  

\begin{Proposition}\label{prop:tau_a* solves}
	$\tau_{a^*}\in\Ec(\R_+)$ solves \eqref{eq:problem2} for all $x\in\R_+$.
\end{Proposition}

\begin{proof}
	Fix $a\in[0,a^*)$. For any $x\le a$, we have $T^x_a \le T^x_{a^*}$. Thus,
	\begin{align*}
	J(x;\Lc\tau_{a^*}(x)) &= \E^{x}\left[\frac{a^*}{1+\beta T^x_{a^*}}\right]= \E^{x}\left[ \E^x\left[\frac{a^*}{1+\beta T^x_{a^*}} \middle|\ \Fc_{T^x_a}\right]  \right]\\
	&\ge \E^{x}\left[\frac{1}{1+\beta T^x_{a}}\E^{a}\left[\frac{a^*}{1+\beta T^a_{a^*}}\right]\right] > \E^{x}\left[\frac{a}{1+\beta T^x_{a}}\right] = J(x;\Lc\tau_a(x)),
	\end{align*}
	where the last inequality follows from Lemma~\ref{lem:eta^a} (iii).
	\qed\end{proof}

The conclusion is twofold. First, it is possible, at least under current setting, to find one single equilibrium policy that solves \eqref{eq:problem2} {\it for all} $x\in\R_+$. Second, this ``optimal'' equilibrium policy $\tau_{a^*}$ is different from $\tau'_{x^*(\widetilde a)}$, the equilibrium generated by the naive policy $\widetilde\tau$ (see Remark~\ref{rem:BES(1) sophisticated}). This indicates that the map $\Theta^* := \lim_{n\to\infty}\Theta^n:\Tc(\X)\mapsto \Ec(\X)$ is in general nonlinear: while $\widetilde\tau\in\Tc(\Tc)$ is constructed from optimal stopping times $\{\widetilde\tau_{x}\}_{x\in\R_+}$ (or ``dynamically optimal'' as in Remark~\ref{rem:PP and ours}), $\Theta^*(\widetilde\tau)=\tau'_{x^*(\widetilde a)}\in\Ec(\X)$ is not optimal under \eqref{eq:problem2}. This is not that surprising once we realize $\widetilde\tau_{x}>\Lc\widetilde\tau(x)>\Lc\tau'_{x^*(\widetilde a)}(x)$ for some $x\in\R_+$. The first inequality is essentially another way to describe time inconsistency, and the second inequality follows from $\ker(\widetilde\tau)
\subset \ker(\Theta\widetilde\tau) = \ker(\tau'_{x^*(\widetilde a)})$.  
It follows that the optimality of $\widetilde\tau_x$ for $\sup_{\tau\in\Tc}J(x;\tau)$ does not necessarily translate to the optimality of $\tau'_{x^*(\widetilde a)}$ for $\sup_{\tau\in\Ec(\R_+)} J(x;\Lc\tau(x))$.



\appendix \normalsize

\section{Proofs for Section~\ref{sec:equilibrium}}\label{sec:proofs for Section 3}

Throughout this appendix, we will constantly use the notation 
\begin{equation}\label{tau_n}
\tau_n := \Theta^n \tau\quad n\in\N,\quad \hbox{for any}\ \tau\in\Tc(\X).
\end{equation}

\subsection{Proof of Proposition~\ref{prop:ttau in Ec}}\label{subsec:ttau in Ec}

Fix $(t,x)\in\X$. We deal with the two cases $\widetilde\tau(t,x) = 0$ and $\widetilde\tau(t,x) = 1$ separately.
If $\widetilde\tau(t,x) = 0$, i.e. $\widetilde\tau_{t,x} = t$, by \eqref{J=esssup J}
\[
g(x) =\sup_{\tau\in\Tc_t}\E^{t,x}[\delta(\tau-t) g(X_\tau)] \ge \E^{t,x}\left[\delta(\Lc^*\widetilde\tau(t,x)-t) g(X^{}_{\Lc^*\widetilde\tau(t,x)})\right],
\]
which implies $(t,x)\in S_{\widetilde\tau}\cup I_{\widetilde\tau}$. We then conclude from \eqref{Theta} that 
\begin{equation*}
\Theta\widetilde{\tau}(t,x)=
\left\{
\begin{array}{ll}
0 & \text{ if }(t,x)\in S_{\widetilde{\tau}} \\
\widetilde\tau(t,x) & \text{ if }(t,x)\in I_{\widetilde{\tau}}
\end{array}
\right.
\ = \widetilde\tau(t,x).
\end{equation*}

If $\widetilde\tau(t,x) =1$, then $\Lc^*\widetilde\tau(t,x) = \Lc\widetilde\tau(t,x) = \inf\{s\ge t :\widetilde\tau(s,X^{t,x}_s)=0 \} = \inf\{s\ge t :\widetilde\tau_{s,X^{t,x}_s}=s \}$. By \eqref{naive} and \eqref{Snell}, $\widetilde\tau_{s,X^{t,x}_s}=s$ means
\[
g(X^{t,x}_s(\omega)) = \esssup_{\tau\in\Tc_s}\E^{s,X^{t,x}_s(\omega)}[\delta(\tau-s) g(X_\tau)],
\]  
which is equivalent to 
\begin{align*}
\delta(s-t) g(X^{t,x}_s(\omega)) &= \delta(s-t) \esssup_{\tau\in\Tc_s}\E^{s,X^{t,x}_s(\omega)}[\delta(\tau-s) g(X_\tau)]\\
&= \esssup_{\tau\in\Tc_s}\E^{s,X^{t,x}_s(\omega)}[\delta(\tau-t) g(X_\tau)] = Z^{t,x}_s(\omega),
\end{align*}  
where the second equality follows from \eqref{identity}. We then conclude that $\Lc^*\widetilde\tau(t,x) = \inf\{s\ge t: \delta(s-t) g(X^{t,x}_s)= Z^{t,x}_s\}=\widetilde\tau_{t,x}$. This, together with \eqref{J=esssup J}, shows that
\begin{align*}
\E^{t,x}\left[\delta(\Lc^*\widetilde\tau(t,x)-t) g(X^{}_{\Lc^*\widetilde\tau(t,x)})\right] &= \E^{t,x}\left[\delta(\widetilde\tau_{t,x}-t) g(X^{}_{\widetilde\tau_{t,x}})\right]\ge g(x), 
\end{align*}
which implies $(t,x)\in I_{\widetilde\tau}\cup C_{\widetilde\tau}$. By \eqref{Theta}, we have  
\begin{equation*}
\Theta\widetilde{\tau}(t,x)=
\left\{
\begin{array}{ll}
\widetilde\tau(t,x) & \text{ if }(t,x)\in I_{\widetilde{\tau}} \\
1 & \text{ if }(t,x)\in C_{\widetilde{\tau}}
\end{array}
\right.
\ = \widetilde\tau(t,x). 
\end{equation*}
We therefore have $\Theta\widetilde\tau(t,x) = \widetilde\tau(t,x)$ for all $(t,x)\in\X$, i.e. $\widetilde\tau\in\Ec(\X)$.

\subsection{Derivation of Proposition~\ref{thm:main 1}}\label{subsec:LTheta^n decreasing}

To prove the technical result Lemma~\ref{lem:g<E} below, we need to introduce shifted random variables as formulated in Nutz \cite{Nutz-2013}. For any $ t\ge 0$ and $\omega\in\Omega$, we define the concatenation of $\omega$ and $\tilde{\omega}\in\Omega$ at time $t$ by
\[
(\omega\otimes_t\tilde{\omega})_s := \omega_s \indi_{[0,t)}(s) + [\tilde\omega_s-(\tilde\omega_t - \omega_t)] \indi_{[t,\infty)} (s),\quad s\ge 0. 
\]
For any $\Fc^{}_\infty$-measurable random variable $\xi:\Omega\mapsto\R$, we define the shifted random variable $[\xi]_{t,\omega}:\Omega\mapsto\R$, which is $\Fc^t_\infty$-measurable, by
\[
[\xi]_{t,\omega} (\tilde\omega):= \xi(\omega\otimes_t \tilde\omega),\quad \forall \tilde\omega\in\Omega. 
\]
Given $\tau\in\Tc$, we write $\omega\otimes_{\tau(\omega)}\tilde{\omega}$ as $\omega\otimes_\tau\tilde{\omega}$, and $[\xi]_{\tau(\omega),\omega} (\tilde\omega)$ as $[\xi]_{\tau,\omega} (\tilde\omega)$. A detailed analysis of shifted random variables can be found in \cite[Appendix A]{BH13-game}; Proposition A.1 therein implies that give $(t,x)\in\X$ fixed, any $\theta\in\Tc_t$ and $\Fc^t_\infty$-measurable $\xi$ with $\E^{t,x}[|\xi|]<\infty$ satisfy 
\begin{equation}\label{cond expect}
\E^{t,x}[\xi\mid\Fc^t_\theta](\omega) = \E^{t,x}\left[[\xi]_{\theta,\omega}\right]\quad \hbox{for a.e. $\omega\in\Omega$}. 
\end{equation}

\begin{Lemma}\label{lem:g<E} 
	For any $\tau\in\Tc(\X)$ and $(t,x)\in\X$, define $t_0:= \Lc^*\tau_1(t,x)\in\Tc_t$ and $s_0 := \Lc^*\tau(t,x)\in\Tc_t$, with $\tau_1$ as in \eqref{tau_n}. If $t_0\le s_0$, then for a.e. $\omega\in\{t <t_0\}$,  
	\[
	g(X^{t,x}_{t_0}(\omega)) \le \E^{t,x}\left[\delta(s_0-t_0)g(X^{}_{s_0})\mid \Fc^t_{t_0}\right] (\omega).
	\]
\end{Lemma}

\begin{proof}
	For a.e. $\omega\in \{t< t_0\}\in\Fc_{t}$, we deduce from $t_0 (\omega) = \Lc^*\tau_1(t,x) (\omega)>t$ that for all $s\in\left(t,t_0(\omega)\right)$ we have $\tau_1(s,X^{t,x}_s(\omega))=1$ . By \eqref{tau_n} and \eqref{Theta}, this implies $(s,X^{t,x}_s(\omega))\notin S_{\tau}$ for all $s\in(t,t_0(\omega))$. Thus,
	\begin{align}\label{1st eqn}
	g(X^{t,x}_s(\omega))  &\le \E^{s,X^{t,x}_s(\omega)}\left[\delta(\Lc^*\tau(s,X^{}_s)-s)g\left(X^{}_{\Lc^*\tau(s,X^{}_s)}\right)\right]\quad \forall s\in\left(t,t_0(\omega)\right).
	\end{align}
	For any $s\in(t,t_0(\omega))$, note that
	$$
	[t_0]_{s,\omega} (\tilde\omega)= t_0(\omega\otimes_s\tilde\omega) = \Lc^*\tau_1(t,x)(\omega\otimes_s\tilde\omega)= \Lc^*\tau_1(s,X^{t,x}_s(\omega))(\tilde\omega),
	$$
	$\forall$ $\tilde\omega\in\Omega$. Since $t_0\le s_0$, similar calculation gives $[s_0]_{s,\omega} (\tilde\omega)= \Lc^*\tau(s,X^{t,x}_s(\omega))(\tilde\omega)$. We thus conclude from \eqref{1st eqn} that  
	\begin{align}\label{1st eqn'}
	g(X^{t,x}_s(\omega))  &\le \E^{s,X^{t,x}_s(\omega)}\left[\delta([s_0]_{s,\omega}-s)g\left([X^{}_{s_0}]_{s,\omega}\right)\right]\nonumber\\
	&\le \E^{s,X^{t,x}_s(\omega)}\left[\delta([s_0]_{s,\omega}-[t_0]_{s,\omega})g\left([X^{}_{s_0}]_{s,\omega}\right)\right],\quad \forall s\in\left(t,t_0(\omega)\right),
	\end{align}
	where the second line holds because $\delta$ is decreasing and also $\delta$ and $g$ are both nonnegative. On the other hand, by \eqref{cond expect}, it holds a.s. that 
	\begin{equation*}
	\E^{t,x}[\delta(s_0-t_0) g(X^{}_{s_0})\mid\Fc^t_s](\omega) = \E^{t,x}\left[\delta([s_0]_{s,\omega}-[t_0]_{s,\omega}) g([X^{t,x}_{s_0}]_{s,\omega})\right]\ \forall s\ge t,\ s\in\Q.
	\end{equation*}
	Note that we used the countability of $\Q$ to obtain the above almost-sure statement. This, together with \eqref{1st eqn'}, shows that 
	it holds a.s. that
	\begin{equation}\label{2nd eqn}
	g(X^{t,x}_s(\omega))\ 1_{\{(t,t_0(\omega))\cap\Q\}}(s) \le \E^{t,x}[\delta(s_0-t_0) g(X^{}_{s_0})\mid\Fc^t_s](\omega)\ 1_{\{(t,t_0(\omega))\cap\Q\}}(s).
	\end{equation}
	Since our sample space $\Omega$ is the canonical space for Brownian motion with the right-continuous Brownian filtration $\F$, the martingale representation theorem holds under current setting. This in particular implies that every martingale has a continuous version. Let $\{M_s\}_{s\ge t}$ be the continuous version of the martingale $\{\E^{t,x}[\delta(s_0-t_0) g(X^{}_{s_0})\mid\Fc^t_s]\}_{s\ge t}$. Then, \eqref{2nd eqn} immediately implies that it holds a.s. that
	\begin{equation}\label{3rd eqn}
	g(X^{t,x}_s(\omega))\ 1_{\{(t,t_0(\omega))\cap\Q\}}(s) \le M_s(\omega)\ 1_{\{(t,t_0(\omega))\cap\Q\}}(s).
	\end{equation}
	Also, using the right-continuity of $M$ and \eqref{cond expect}, one can show that for any $\tau\in\Tc_t$, $M_\tau = \E^{t,x}[\delta(s_0-t_0) g(X^{}_{s_0})\mid\Fc^t_\tau]$ a.s. Now, we can take some $\Omega^*\in\Fc_\infty$ with $\P[\Omega^*] =1$ such that for all $\omega\in\Omega^*$, \eqref{3rd eqn} holds true and $M_{t_0}(\omega) = \E^{t,x}[\delta(s_0-t_0) g(X^{}_{s_0})\mid\Fc^t_{t_0}](\omega)$. For any $\omega\in\Omega^*\cap\{t<t_0\}$, take $\{k_n\}\subset\Q$ such that $k_n >t$ and $k_n\uparrow t_0(\omega)$. Then, \eqref{3rd eqn} implies  
	$
	g(X^{t,x}_{k_n}(\omega)) \le M_{k_n}(\omega),\ \forall n\in\N. 
	$
	As $n\to\infty$, we obtain from the continuity of $s\mapsto X_s$ and $z\mapsto g(z)$, and the left-continuity of $s\mapsto M_s$ that 
	$
	g(X^{t,x}_{t_0}(\omega))  \le M_{t_0}(\omega) = \E^{t,x}[\delta(s_0-t_0) g(X^{}_{s_0})\mid\Fc^t_{t_0}](\omega). 
	$
	\qed\end{proof}

Now, we are ready to prove Proposition~\ref{thm:main 1}.

\begin{proof}[Proof of Proposition~\ref{thm:main 1}] 
	We will prove \eqref{ker increasing} by induction. We know that the result holds for $n=0$ by \eqref{1st iteration decreasing}. Now, assume that \eqref{ker increasing} holds for $n=k\in \N\cup\{0\}$, and we intend to show that \eqref{ker increasing} also holds for $n=k+1$. Recall the notation in \eqref{tau_n}. Fix $(t,x)\in\ker(\tau_{k+1})$, i.e. $\tau_{k+1}(t,x)=0$. If $\Lc^*\tau_{k+1}(t,x) = t$, then $(t,x)$ belongs to $I_{\tau_{k+1}}$. By \eqref{Theta},  we get $\tau_{k+2}(t,x) = \Theta\tau_{k+1}(t,x)= \tau_{k+1}(t,x) =0$, and thus $(t,x)\in\ker(\tau_{k+2})$, as desired. We therefore assume below that $\Lc^*\tau_{k+1}(t,x)>t$.
	
	By \eqref{Theta}, $\tau_{k+1}(t,x)=0$ implies 
	\begin{equation}\label{hehe}
	g(x)\ge \E^{t,x}[\delta(\Lc^*\tau_k(t,x)-t) g(X^{}_{\Lc^*\tau_k(t,x)})].
	\end{equation}
	Let $t_0 := \Lc^*\tau_{k+1}(t,x)$ and $s_0 := \Lc^*\tau_k(t,x)$. Under the induction hypothesis $\ker(\tau_k)\subseteq\ker(\tau_{k+1})$, we have $t_0\le s_0$, as $t_0$ and $s_0$ are hitting times to $\ker(\tau_{k+1})$ and $\ker(\tau_{k})$, respectively; see \eqref{L*}. Using \eqref{hehe}, $t_0\le s_0$, Assumption~\ref{asm:DI}, and $g$ being nonnegative, 
	\begin{align*}
	g(x) &\ge  \E^{t,x}[\delta(s_0-t) g(X^{}_{s_0})] \ge \E^{t,x}[\delta(t_0-t)\delta(s_0-t_0) g(X^{}_{s_0})]\\
	&=\E^{t,x}\left[\delta(t_0-t)\E^{t,x}\left[\delta(s_0-t_0)g(X^{}_{s_0})\mid\Fc^t_{t_0}\right]\right]\\
	&\ge \E^{t,x}\left[\delta(t_0-t)g(X^{}_{t_0})\right],
	\end{align*}
	where the second line follows from the tower property of conditional expectations, and the third line is due to Lemma~\ref{lem:g<E}. This implies $(t,x)\notin C_{\tau_{k+1}}$, and thus
	\begin{equation}\label{lala}
	\tau_{k+2}(t,x) = \left\{
	\begin{array}{ll}
	0 & \text{ for }(t,x)\in S_{\tau_1} \\
	\tau_{k+1} (t,x) & \text{ for }(t,x)\in I_{\tau_1} 
	\end{array}
	\right. \ = 0.
	\end{equation}
	That is, $(t,x)\in\ker(\tau_{k+2})$. Thus, we conclude that $\ker(\tau_{k+1})\subseteq\ker(\tau_{k+2})$, as desired. 
	
	It remains to show that $\tau_0$ defined in \eqref{tau0} is a stopping policy. Observe that for any $(t,x)\in\X$, $\tau_0(t,x) =0$ if and only if $\Theta^n\tau(t,x)=0$, i.e. $(t,x)\in\ker(\Theta^n\tau)$, for $n$ large enough. This, together with \eqref{ker increasing}, implies that
	\[
	\{(t,x)\in\X : \tau_0(t,x)=0 \} = \bigcup_{n\in\N}\ker(\Theta^n\tau)\in \Bc(\X). 
	\]  
	Hence, $\tau_0:\X\mapsto\{0,1\}$ is Borel measurable, and thus an element in $\Tc(\X)$.
	\qed\end{proof}


\subsection{Proof of Proposition~\ref{prop:naive satisfies}}\label{subsec:naive satisfies}

Fix $(t,x)\in\ker(\widetilde\tau)$. Since $\widetilde\tau(t,x)=0$, i.e. $\widetilde\tau_{t,x}=t$, \eqref{naive}, \eqref{Snell}, and \eqref{J=esssup J} imply
\[
g(x) = \sup_{\tau\in\Tc_t} \E^{t,x}[\delta(\tau-t) g(X_\tau)] \ge \E^{t,x} \left[\delta(\Lc^*\widetilde\tau(t,x)-t) g(X_{\Lc^*\widetilde\tau(t,x)})\right].
\]
This shows that $(t,x)\in S_{\widetilde\tau}\cup I_{\widetilde\tau}$. Thus, we have $\ker(\widetilde\tau)\subseteq S_{\widetilde\tau}\cup I_{\widetilde\tau}$. It follows that
\[
\ker(\widetilde\tau) = (\ker(\widetilde\tau)\cap S_{\widetilde\tau}) \cup (\ker(\widetilde\tau)\cap I_{\widetilde\tau}) \subseteq S_{\widetilde\tau} \cup (\ker(\widetilde\tau)\cap I_{\widetilde\tau}) = \ker(\Theta\widetilde\tau),
\]
where the last equality follows from \eqref{Theta}.

\subsection{Derivation of Theorem~\ref{thm:main 2}}\label{subsec:proof of lem:technical}

\begin{Lemma}\label{lem:technical}
	Suppose Assumption~\ref{asm:DI} holds and $\tau\in\Tc(\X)$ satisfies \eqref{1st iteration decreasing}. Then $\tau_0$ defined in \eqref{tau0} satisfies
	\[
	\Lc^*\tau_0(t,x) = \lim_{n\to\infty} \Lc^*\Theta^n\tau(t,x),\quad \forall (t,x)\in\X.
	\]
\end{Lemma}

\begin{proof}
	We will use the notation in \eqref{tau_n}. Recall that $\ker(\tau_{n})\subseteq\ker(\tau_{n+1})$ for all $n\in\N$ and $\ker(\tau_0) = \bigcup_{n\in\N}\ker(\tau_n)$ from Proposition~\ref{thm:main 1}. By \eqref{L*}, this implies that $\{\Lc^*\tau_n(t,x)\}_{n\in\N}$ is a nonincreasing sequence of stopping times, and
	\[
	\Lc^*\tau_0(t,x) \le t_0:=\lim_{n\to\infty} \Lc^*\tau_n(t,x).
	\]
	It remains to show that $\Lc^*\tau_0(t,x) \ge t_0$. We deal with the following two cases.
	
	(i) On $\{\omega\in\Omega : \Lc^*\tau_0(t,x)(\omega)=t\}$: By \eqref{L*}, there must exist a sequence $\{t_m\}_{m\in\N}$ in $\R_+$, depending on $\omega\in\Omega$, such that $t_m\downarrow t$ and $\tau_0(t_m, X^{t,x}_{t_m}(\omega)) = 0$ for all $m\in\N$. For each $m\in\N$, by the definition of $\tau_0$ in \eqref{tau0}, there exists $n^*\in\N$ large enough such that $\tau_{n^*}(t_m, X^{t,x}_{t_m}(\omega)) = 0$, which implies $\Lc^*\tau_{n^*}(t,x)(\omega)\le t_m$. Since $\{\Lc^*\tau_n(t,x)\}_{n\in\N}$ is nonincreasing, we have $t_0(\omega)\le \Lc^*\tau_{n^*}(t,x)(\omega)\le t_m$. With $m\to\infty$, we get $t_0(\omega)\le t=\Lc^*\tau_0(t,x)(\omega)$.  
	
	(ii) On $\{\omega\in\Omega : \Lc^*\tau_0(t,x)(\omega)>t\}$: Set $s_0:=\Lc^*\tau_0(t,x)$. If $\tau_0(s_0(\omega), X^{t,x}_{s_0}(\omega))=0$, then by \eqref{tau0} there exists $n^*\in\N$ large enough such that $\tau_{n^*}(s_0(\omega), X^{t,x}_{s_0}(\omega))=0$. Since $\{\Lc^*\tau_n(t,x)\}_{n\in\N}$ is nonincreasing, $t_0(\omega)\le \Lc^*\tau_{n^*}(t,x)(\omega)\le s_0(\omega)$, as desired. If $\tau_0(s_0(\omega), X^{t,x}_{s_0}(\omega))=1$, then by \eqref{L*} there exist a sequence $\{t_m\}_{m\in\N}$ in $\R_+$, depending on $\omega\in\Omega$, such that $t_m\downarrow s_0(\omega)$ and $\tau_0(t_m, X^{t,x}_{t_m}(\omega)) = 0$ for all $m\in\N$. Then we can argue as in case (i) to show that $t_0(\omega)\le s_0(\omega)$, as desired.
	\qed\end{proof}

Now, we are ready to prove Theorem~\ref{thm:main 2}.

\begin{proof}[Proof of Theorem~\ref{thm:main 2}]
	By Proposition~\ref{thm:main 1}, $\tau_0\in\Tc(\X)$ is well-defined. For simplicity, we will use the notation in \eqref{tau_n}. Fix $(t,x)\in\X$. If $\tau_0(t,x)=0$, by \eqref{tau0} we have $\tau_n(t,x)=0$ for $n$ large enough. Since $\tau_{n}(t,x) = \Theta\tau_{n-1}(t,x)$, we deduce from ``$\tau_n(t,x)=0$ for $n$ large enough'' and \eqref{Theta} that $(t,x)\in S_{\tau_{n-1}}\cup I_{\tau_{n-1}}$ for $n$ large enough. That is, 
	$
	g(x)\ge \E^{t,x}\left[\delta(\Lc^*\tau_{n-1}(t,x)-t) g(X_{\Lc^*\tau_{n-1}(t,x)})\right]\ \hbox{for $n$ large enough}.
	$   
	With $n\to\infty$, the dominated convergence theorem and Lemma~\ref{lem:technical} yield 
	\[
	g(x)\ge \E^{t,x}\left[\delta(\Lc^*\tau_{0}(t,x)-t) g(X_{\Lc^*\tau_{0}(t,x)})\right],
	\]   
	which shows that $(t,x)\in S_{\tau_0}\cup I_{\tau_0}$. We then deduce from \eqref{Theta} and $\tau_0(t,x)=0$ that $\Theta\tau_0(t,x) = \tau_0(t,x)$. On the other hand, if $\tau_0(t,x)=1$, by \eqref{tau0} we have $\tau_n(t,x)=1$ for $n$ large enough. Since $\tau_{n}(t,x) = \Theta\tau_{n-1}(t,x)$, we deduce from ``$\tau_n(t,x)=1$ for $n$ large enough'' and \eqref{Theta} that $(t,x)\in C_{\tau_{n-1}}\cup I_{\tau_{n-1}}$ for $n$ large enough. That is, 
	$$
	g(x)\le \E^{t,x}\left[\delta(\Lc^*\tau_{n-1}(t,x)-t) g(X_{\Lc^*\tau_{n-1}(t,x)})\right]\ \hbox{for $n$ large enough}.
	$$
	With $n\to\infty$, the dominated convergence theorem and Lemma~\ref{lem:technical} yield   
	\[
	g(x)\le \E^{t,x}\left[\delta(\Lc^*\tau_{0}(t,x)-t) g(X_{\Lc^*\tau_{0}(t,x)})\right],
	\]   
	which shows that $(t,x)\in C_{\tau_0}\cup I_{\tau_0}$. We then deduce from \eqref{Theta} and $\tau_0(t,x)=1$ that $\Theta\tau_0(t,x) = \tau_0(t,x)$.
	We therefore conclude that $\tau_0\in\Ec(\X)$. 
	\qed\end{proof}


\section{Proofs for Section~\ref{sec:examples}}

\subsection{Derivation of Proposition~\ref{prop:ttau explicit}}\label{subsec:ttau explicit}

In the classical case of exponential discounting, \eqref{identity} ensures that for all $s\ge 0$,
\begin{equation}\label{=supermart.}
\delta(s)v(X_s^{x})=\sup_{\tau\in\Tc}\E^{X^{x}_s}\left[\delta(s+\tau) g(X^{}_\tau)\right]= \sup_{\tau\in\Tc_s}\E^{x}\left[\delta(\tau) g(X^{}_\tau) \mid \Fc_s\right],
\end{equation}
which shows that $\{\delta(s)v(X_s^{x})\}_{s\ge 0}$ is a supermartingale. Under hyperbolic discounting \eqref{hyperbolic}, since $\delta(r_1)\delta(r_2)<\delta(r_1+r_2)$ for all $r_1,r_2\ge 0$, $\{\delta(s)v(X_s^{x})\}_{s\ge t}$ may no longer be a supermatingale, as the first equality in the above equation fails. 

To overcome this, we introduce the auxiliary value function: for $(s,x)\in \R^2_+$,
\begin{align}\label{eq:objective function*}
V(s,x)&:=\sup_{\tau\in\Tc}\E^{x}\left[\delta(s+\tau) g(X^{}_\tau)\right] =\sup_{\tau\in\Tc}\E^{x}\left[\frac{X^{}_\tau}{1+\beta(s+\tau)}\right].
\end{align}
By definition, $V(0,x)=v(x)$, and $\{V(s,X^{x}_s) \}_{s\ge 0}$ is a supermartingale as $V(s,X^{x}_s)$ is equal to the right hand side of \eqref{=supermart.}.

\begin{proof}[Proof of Proposition~\ref{prop:ttau explicit}]
	Recall that $X_s=|W_s|$ for a one-dimensional Brownian motion $W$. Let $y\in\R$ be the initial value of $W$, and define $\bar V(s,y) := V(s,|y|)$. The associated variational inequality for $\bar V(s,y)$ is the following: for $(s,y)\in [0,\infty)\times\R$,
	\begin{equation}\label{variational}
	\min\left\{w_s(s,y)+\frac{1}{2}w_{yy}(s,y),\ w(s,y)-\frac{|y|}{1+\beta s}\right\} =0.
	\end{equation} 
	Taking $s\mapsto b(s)$ as the free boundary to be determined, we can rewrite \eqref{variational} as \begin{equation} \label{pde v}
	\begin{cases}
	w_{s}(s,y) + \frac{1}{2}w_{yy}(s,y)=0,\ \  w(s,y)>\frac{|y|}{1+\beta s},\ \ \ &\hbox{for}\ |y|< b(s);\\
	w(s,y)=\frac{|y|}{1+\beta s},\ \ \ &\hbox{for}\ |y|\ge b(s).
	\end{cases}
	\end{equation}
	Following \cite{pederson2000solving}, we propose the ansatz
	$ w(s,y)=\frac{1}{\sqrt{1+\beta s}} h(\frac{y}{\sqrt{1+\beta s}}). $
	Equation \eqref{pde v} then becomes a one-dimensional free boundary problem:
	\begin{equation}
	\label{ode h}
	\left\{
	\begin{array}{ll}
	-\beta zh'(z)+h''(z)=\beta h(z),\ \ h(z)>|z|,\ \ &\quad \hbox{for}\ |z|<\frac{b(s)}{\sqrt{1+\beta s}};\\
	h(z)=|z|,\ \ &\quad \hbox{for} \ |z|\ge \frac{b(s)}{\sqrt{1+\beta s}}.
	\end{array}
	\right.
	\end{equation}
	Since the variable $s$ does not appear in the above ODE, we take $b(s) = \alpha \sqrt{1+ \beta s}$ for some $\alpha \ge 0$.
	The general solution of the first line of \eqref{ode h} is
	\begin{equation*}
	h(z) = e^{\frac{\beta}{2}z^2}\left(c_1+c_2 \sqrt{\frac{2}{\beta}}\int_0^{\sqrt{{\beta}/{2}} z} e^{-u^2}du\right),\quad
	(c_1, c_2) \in \R^2\; .
	\end{equation*}
	The second line of \eqref{ode h} gives $h(\alpha) =\alpha$. 
	We then have
	\begin{equation*}
	w(s,y) = 
	\begin{cases}
	\frac{e^{\frac{\beta y^2}{2(1+\beta s)}}}{\sqrt{1+\beta s}}\left(c_1 + c_2 \sqrt{\frac{2}{\beta}}\int_0^{\frac{\sqrt{{\beta}/{2}} y}{\sqrt{1+\beta s}}} e^{-u^2}du\right), & |y|< \alpha \sqrt{1+\beta s};\\
	\frac{|y|}{1+ \beta s}, & |y|\ge \alpha \sqrt{1+\beta s}.
	\end{cases}
	\end{equation*}
	To find the parameters $c_1, c_2$ and $\alpha$, we equate the partial derivatives of $(s,y)\mapsto w(s,y)$ obtained on both sides of the free boundary. This yields the equations
	\begin{equation*}
	\label{eq:system conditions}
	\alpha = e^{\frac{\beta}{2}\alpha^2}\left(c_1+c_2\sqrt{\frac{2}{\beta}}\int_{0}^{\sqrt{{\beta}/{2}}\alpha}e^{-u^2}du\right) \quad \text{and}\quad \text{sgn}(x)-c_2 = \text{sgn}(x)\alpha^2 \beta.
	\end{equation*}
	The last equation implies $c_2=0$. This, together with the first equation, shows that $\alpha = 1/\sqrt{\beta}$ and $c_1 = \alpha e^{-1/2}$.
	Thus, we obtain
	\begin{equation}
	\label{eq:good v}
	w(s,y) = 
	\begin{cases}
	\frac{1}{\sqrt{\beta}\sqrt{1+\beta s}}\exp\left(\frac{1}{2}\left(\frac{\beta y^2}{1+\beta s}-1\right)\right), & |y|<\sqrt{1/\beta+s},\\
	\frac{|y|}{1+ \beta s},  & |y|\ge\sqrt{1/\beta+s}.
	\end{cases}
	\end{equation}
	Note that $w(s,y)> \frac{|y|}{1+ \beta s}$ for $|y|<\sqrt{1/\beta+s}$. Indeed, by defining the function $h(y) := \frac{1}{\sqrt{\beta}\sqrt{1+\beta s}}\exp\left(\frac{1}{2}\left(\frac{\beta y^2}{1+\beta s}-1\right)\right)-\frac{y}{1+\beta s}$ and observing that $h(0)>0$, $h(\sqrt{1/\beta+s})=0$, and $h'(y)<\frac{1}{1+\beta s} -\frac{1}{1+\beta s}=0$ for all $y\in(0,\sqrt{1/\beta+s})$, we conclude $h(y)>0$ for all $y\in[0,\sqrt{1/\beta+s})$, or $w(s,y)>  \frac{|y|}{1+ \beta s}$ for $ |y|<\sqrt{1/\beta+s}$. Also note that $w$ is $\Cc^{1,1}$ on $[0,+\infty)\times \R$, and $\Cc^{1,2}$ on the domain
	$\{(s,y)\in [0,\infty)\times\R ~:~ |y|<\sqrt{1/\beta+s }\}.$ 
	Moreover, by \eqref{eq:good v}, $w_s(s,y) + \frac{1}{2}w_{yy}(s,y)<0$ for $|y|>\sqrt{1/\beta + s)}$. 
	We then conclude from the standard verification theorem (see e.g. \cite[Theorem 3.2]{OS05})  that $\bar V(s,y) = w(s,y)$ is a smooth solution of \eqref{pde v}. This implies that $\{\bar V(s,W^{y}_s)\}_{s\ge 0}$ is a supermartingale, and $\{\bar V(s\wedge\tau^*_{y},W^{y}_{s\wedge\tau^*_{y}})\}_{s\ge 0}$ is a true martingale, with $\tau^*_{y} := \inf\{s\ge 0: |W^y_s|\ge \sqrt{1/\beta+s}\}$. 
	It then follows from standard arguments that $\tau^*_y$ is the smallest optimal stopping time of $\bar V(0,y)$, and thus $\hat \tau_{x}:=\inf\{s\ge 0: X^x_s\ge \sqrt{1/\beta+s}\}$ is the smallest optimal stopping time of \eqref{eq:objective function example}. 
	In view of Proposition \ref{prop:standard result},  $\widetilde{\tau}_{x} = \hat \tau_{x}$.
	\qed\end{proof}

\subsection{Proof of Lemma \ref{lem:tau_a}}\label{app:threshold}
First, we prove that $E$ is totally disconnected. 
If $\ker(\tau)=[a,\infty)$, then $E=\emptyset$ and there is nothing to prove. Assume that there exists $x^*> a$ such that $x^*\notin \ker(\tau)$. Define
\begin{equation*}
\ell := \sup \brace{b\in \ker(\tau) ~:~ b< x^*} \And u:=\inf \brace{b\in \ker(\tau) ~:~ b>x^*}.
\end{equation*}
We claim that $\ell=u=x^*$. Assume to the contrary $\ell<u$. Then $\tau(x)=1$ for all $x\in (\ell,u)$. Thus, given $y\in(\ell,u)$, $\Lc^*\tau(y) = T^y_{} :=\inf\{s\ge 0 : X^y_s \notin (\ell,u)\}>0$, and
\begin{equation}\label{J<y}
J(y; \Lc^*\tau(y)) = \E^{y}\left[\frac{X_{T^y_{}}}{1+\beta T^y_{}}\right] < \E^{y}[X_{T^y_{}}] = \ell \P[X_{T^y_{}}=\ell]+ u\P[X_{T^y_{}}=u].
\end{equation}
Since $X_s=|W_s|$ for a one-dimensional Brownian motion $W$ and $0<\ell<y<u$, by the optional sampling theorem $\P[X_{T^y_{}}=\ell] = \P[W^y_{s}\ \hbox{hits $\ell$ before hitting $u$}] = \frac{u-y}{u-\ell}$ and $\P[X_{T^y_{}}=u]=\P[W^y_{s}\ \hbox{hits $u$ before hitting $\ell$}] =\frac{y-\ell}{u-\ell}$. This, together with \eqref{J<y}, gives $J(y; \Lc^*\tau(y)) < y$. This implies $y\in S_\tau$, and thus $\Theta\tau(y)=0$ by \eqref{Theta'}. Then $\Theta\tau(y)\neq \tau(y)$, a contradiction to $\tau\in\Ec(\R_+)$. This already implies that  $E$ is totally disconnected, and thus $\overline{\ker(\tau)}=[a,\infty)$. The rest of the proof follows from Lemma~\ref{lem:closure ker}.

\subsection{Proof of Lemma~\ref{lem:eta^a}}
\label{subsec:lem eta^a}

(i) Given $a\ge 0$, it is obvious from definition that $\eta(0,a)\in(0,a)$ and $\eta(a,a)=a$. 
Fix $x\in(0,a)$, and let $f^{x}_a$ denote the density of $T^x_a$. We obtain
\begin{equation}\label{E[1/1+T]}
\begin{split}
\E^{x}\left[\frac{1}{1+\beta T^x_a}\right] &= \int_0^\infty\frac{1}{1+\beta t}f^x_a(t)dt = \int_0^\infty \int_0^\infty e^{-(1+\beta t)s}f^x_a(t)ds\ dt\\
&= \int_0^\infty e^{-s}\left(\int_0^\infty e^{-\beta st}f^x_a(t)dt\right)\ ds = \int_0^\infty e^{-s} \E^{x}[e^{-\beta sT^x_a}] ds.
\end{split}
\end{equation}
Since $T^x_a$ is the first hitting time of a one-dimensional Bessel process, we compute its Laplace transform using Theorem 3.1 of \cite{Kent78} (or Formula 2.0.1 on p. 361 of \cite{BS-book-2002}):
\begin{equation}\label{Laplace}
\E^{x}\left[e^{-\frac{\lambda^2}{2} T^x_a}\right] = \frac{\sqrt{x} I_{-\frac12}(x\lambda)}{\sqrt{a} I_{-\frac12}(a\lambda)}= \cosh(x\lambda)\sech(a\lambda),\ \ \quad \hbox{for}\ x\le a.
\end{equation}
Here, $I_\nu$ denotes the modified Bessel function of the first kind. 
Thanks to the above formula with $\lambda=\sqrt{2\beta s}$, we obtain from \eqref{E[1/1+T]} that
\begin{equation}\label{eta^a formula}
\eta(x,a) = a \int_0^\infty e^{-s} \cosh(x\sqrt{2\beta s})\sech(a\sqrt{2\beta s}) ds.
\end{equation}
It is then obvious that $x\mapsto\eta(x,a)$ is strictly increasing. Moreover, 
\[
\eta_{xx}(x,a) = 2a\beta^2 \int_0^\infty e^{-s} s\cosh(x\sqrt{2\beta s})\sech(a\sqrt{2 \beta s}) ds>0\ \quad \hbox{for}\ x\in[0,a],
\]
which shows the strict convexity.

(ii) This follows from \eqref{eta^a formula} and the dominated convergence theorem.

(iii) {We will first prove the desired result with $x^*(a)\in (0,a)$, and then upgrade it to $x^*(a)\in (0,a^*)$.} Fix $a\ge 0$. In view of the properties in (i), 
we observe that the two curves $y=\eta(x,a)$ and $y=x$ intersect at some $x^*(a)\in(0,a)$ if and only if $\eta_x(a,a)>1$. 
Define $k(a):=\eta_x(a,a)$. 
By \eqref{eta^a formula}, 
\begin{equation}\label{k}
k(a)=a\int_0^\infty e^{-s} \sqrt{2\beta s}\tanh(a\sqrt{2\beta s}) ds.
\end{equation}
Thus, we see that $k(0)=0$ and $k(a)$ is strictly increasing on $(0,1)$, since for any $a>0$,
$$k'(a)=\int_0^\infty e^{-s} \sqrt{2s}\left(\tanh(a\sqrt{2s})+\frac{a\sqrt{2s}}{\cosh^2(a\sqrt{2s})}\right) ds >0.$$
By numerical computation, 
$k(1/\sqrt{\beta}) =\int_0^\infty e^{-s} \sqrt{2s}\tanh(\sqrt{2s}) ds \approx 1.07461 >1.$ 
It follows that there must exist $a^*\in(0,1/\sqrt{\beta})$ such that $k(a^*)=\eta_{x}(a^*,a^*)=1$. Monotonicity of $k(a)$ then gives the desired result.

Now, for any $a> a^*$, we intend to upgrade the previous result to $x^*(a)\in(0,a^*)$. Fix $x\ge 0$. 
By the definition of $\eta$ and (ii), on the domain $a\in[x,\infty)$, the map $a\mapsto \eta(x,a)$ must either first increases and then decreases to $0$, or directly decreases down to $0$. From \eqref{eta^a formula}, 
we have
$$
\eta_a(x,x) = 1-x\int_0^\infty e^{-s} \sqrt{2\beta s}\tanh(x\sqrt{2\beta s}) ds = 1-k(x),
$$
with $k$ as in \eqref{k}. Recalling $k(a^*)=1$, we have $\eta_a(a^*,a^*)=0$. Notice that
\begin{align*}
\eta_{aa}(a^*,a^*) &= -\frac{2}{a^*}k(a^*) -2 \beta a^* + a^*\int_0^\infty 4\beta s e^{-s}\tanh^2(a^*\sqrt{2\beta s}) ds\\
&\le  -\frac{2}{a^*} + 2\beta a^* <0,
\end{align*}
where the second line follows from $\tanh(x)\le 1$ for $x\ge 0$ and $a^*\in (0,1/\sqrt{\beta})$. Since $\eta_a(a^*,a^*)=0$ and $\eta_{aa}(a^*,a^*)<0$, we conclude that on the domain $a\in [a^*,\infty)$, the map $a\mapsto\eta(a^*,a)$ decreases down to $0$. Now, for any $a>a^*$, since $\eta(a^*,a) < \eta(a^*,a^*)= a^*$, we must have $x^*(a)<a^*$.

\bibliographystyle{siam}
\bibliography{refs}

\end{document}